\documentclass[a4paper,oneside,10pt]{article}%
\usepackage{amsmath}
\usepackage{amsfonts}
\usepackage{amssymb}
\usepackage{graphicx}
\usepackage{color}
\usepackage[square,numbers,sort&compress]{natbib}%
\providecommand{\U}[1]{\protect\rule{.1in}{.1in}}

\newtheorem{theorem}{Theorem}[section]

\newtheorem{example}[theorem]{Example}

\newtheorem{lemma}[theorem]{Lemma}

\newtheorem{remark}[theorem]{Remark}

\newenvironment{proof}[1][Proof]{\noindent\textbf{#1.} }{\ \rule{0.5em}{0.5em}}
\numberwithin{equation}{section}
\begin{document}

\title{Explicit $\theta$-Schemes for Solving Anticipated Backward Stochastic
Differential Equations }
\author{Mingshang Hu \thanks{Zhongtai Securities Institute for Financial Studies,
Shandong University, Jinan, Shandong 250100, PR China. humingshang@sdu.edu.cn.
Research supported by NSF (No. 11671231) and Young Scholars Program of
Shandong University (No. 2016WLJH10). }
\and Lianzi Jiang\thanks{Zhongtai Securities Institute for Financial Studies,
Shandong University, Jinan, Shandong 250100, PR China. jianglianzi95@163.com.}}
\maketitle

\textbf{Abstract}.
In this paper, a class of stable explicit $\theta$-schemes are proposed for
solving anticipated backward stochastic differential equations (anticipated BSDEs) which generator
not only contains the present values of the solutions but also the future. We
subtly transform the delay process of the generator into the current
measurable process, resulting in high-order convergence rate. We also analyze
the stability of our numerical schemes and strictly prove the error estimates.
Various numerical tests powerful demonstrate high accuracy of the proposed
numerical schemes.



\section{Introduction\label{sec1}}

We consider the anticipated backward stochastic differential equations
(anticipated BSDEs) on a complete filtered probability space $\left(
\Omega,\mathcal{F},\mathbb{F},P\right)  $, which generator contains not only
the present values of $\left(  Y_{\cdot},Z_{\cdot}\right)  $, but also the
future. The general form of the anticipated BSDEs is as follow:
\begin{equation}
\left\{
\begin{array}
[c]{ll}%
-dY_{t}=\mathbf{f}\left(  t,Y_{t},Z_{t},Y_{t+\delta(t)},Z_{t+\zeta(t)}\right)
dt-Z_{t}dW_{t}, & t\in\left[  0,T\right]  ;\\
Y_{t}=\xi_{t}, & t\in\left[  T,T+S\right]  ;\\
Z_{t}=\eta_{t}, & t\in\left[  T,T+S\right]  ,
\end{array}
\right.  \label{1.1}%
\end{equation}
where $\mathbb{F=(}\mathcal{F}_{t}\mathcal{)}_{t\geq0}$ is the natural
filtration generated by the standard $d$-dimensional Brownian motion $W=$
$(W_{t})_{t\geq0}$, $\delta(\cdot)$ and $\zeta(\cdot)$ are two $\mathbb{R}%
^{+}$-valued continuous functions defined on $[0,T]$, satisfying that

\begin{enumerate}
\item[(i)] there exists a constant $S>0$ such that, for all $t\in\left[
0,T\right]  ,$
\begin{equation*}
\begin{array}
[c]{ll}
t+\delta\left(  t\right)  \leq T+S;\text{ \ }\text{ \ } & t+\zeta(t)\leq T+S,
\end{array}
\label{cond1}
\end{equation*}

\item[(ii)] there exists a constant $L_{0}>0$ such that, for all $t\in\left[
0,T\right]  $ and for all nonnegative and integrable $g(\cdot)$,
\begin{equation*}
\begin{array}
[c]{l}
\int_{t}^{T}g\left(  s+\newline\delta\left(  s\right)  \right)  ds\leq
L_{0}\int_{t}^{T+S}g\left(  s\right)  ds;\\
\int_{t}^{T}g\left(  s+\zeta\left(  s\right)  \right)  ds\leq L_{0}\int
_{t}^{T+S}g\left(  s\right)  ds.
\end{array}
\label{cond2}
\end{equation*}

\end{enumerate}
The generator $\mathbf{f(}t,\omega,y,z,\xi,\eta):\Omega\times\mathbb{R}
^{m}\times\mathbb{R}^{m\times d}\times L^{2}(\mathcal{F}_{r};\mathbb{R}
^{m})\times L^{2}(\mathcal{F}_{r^{\prime}};\mathbb{R}^{m\times d})\rightarrow
L^{2}(\mathcal{F}_{t};\mathbb{R}^{m})$ is $\mathcal{F}_{t}$-measurable, for
all $t\in\left[  0,T\right]  $, $r,r^{\prime}\in\lbrack t,T+S]$. $\xi_{t}$ and
$\eta_{t}$, $t\in\lbrack T,T+S]$ are given terminal conditions. The pair of
$(Y_{\cdot},Z_{\cdot})$ is called $L^{2}$-solution of the anticipated BSDEs if
it is $\mathcal{F}_{t}$-adapted and square integrable processes.

In 1990, Pardoux and Peng \cite{PP1990} first established the framework for
the nonlinear BSDEs and rigorous proved the existence and uniqueness results
for the solution, their extraordinary works are continued in
\cite{KPQ1997,P1991,P1999,PP1994}. In 2009, Peng and Yang \cite{PY2009} introduced a
new type of BSDE, namely anticipated BSDEs $\left(  \ref{1.1}\right)  $. They
strictly proved the existence and uniqueness of solutions and extended the
duality between stochastic differential delay equations (SDDEs) and anticipated BSDEs to stochastic control problems.
Since then, a large amount of related research is booming,
\cite{BCY2015,BCY2016,CW2010,OS2011} etc..

Since analytical solutions of the backward stochastic differential equations
are often difficult to obtain, the numerical approximation is indispensable.
Many works have be done on the numerical computing of BSDEs, including one
step schemes, such as
\cite{B1997,BD2007,CM,GL2010,GLW2005,MP2002,MT2006,MT2007,Zh2004,ZhCP2006,ZhLZh2012,ZhWP2009}
and multi-step schemes \cite{C2013,ZhFZh2014,ZhZhJ2010,ZhZhJ2016}, but little
attention has been paid to the anticipated BSDEs.

In this paper, we just discuss the special case of the anticipated BSDEs in
the following form:
\begin{equation}
\left\{
\begin{array}
[c]{ll}
-dY_{t}=\mathbb{E}\left[  f\mathbf{(}t,Y_{t},Z_{t},Y_{t+\delta(t)}
,Z_{t+\delta(t)})|\mathcal{F}_{t}\right]  dt-Z_{t}dW_{t}, & t\in\left[
0,T\right]  ;\\
Y_{t}=\xi_{t}, & t\in\left[  T,T+S\right]  ;\\
Z_{t}=\eta_{t}, & t\in\left[  T,T+S\right]  ,
\end{array}
\right.  \label{1.2}
\end{equation}
where the generator $\mathbf{f}=\mathbb{E}[f(t,y_{1},z_{1},y_{2}
,z_{2})|\mathcal{F}_{t}]$, $f:\lbrack0,T]\times\Omega\times\mathbb{R}
^{m}\times\mathbb{R}^{m\times d}\times\mathbb{R}^{m}\times\mathbb{R}^{m\times
d}\rightarrow\mathbb{R}^{m}$ is an $\mathbb{R}^{m}$-valued vector function,
$\delta\left(  t\right)  $ is a $\mathbb{R}^{+}$-valued monotonically
increasing continuous function. Moreover, we can easily verify that
$\delta\left(  t\right)  $ satisfies (i) and (ii), the anticipated BSDEs $\left(  \ref{1.2}
\right)  $ satisfy the conditions of existence and uniqueness Theorem in
\cite{PY2009}, thus have a pair of unique $L^{2}$-solution $\left(
Y_{t},Z_{t}\right)  $. Assume that under some regularity conditions about
$\xi_{t}=\varphi\left(  t,W_{t}\right)  $, $f$ and $\delta\left(  t\right)  $,
the unique solution $\left(  Y_{t},Z_{t}\right)  $ of $\left(  \ref{1.2}
\right)  $ can be represented as (see as \cite{KPQ1997,LSU1968,MZ2002,P1991})
\begin{equation}
Y_{t}=u\left(  t,W_{t}\right)  ,\;\;Z_{t}=\nabla_{x}u\left(  t,W_{t}\right)
,\;\;\forall t\in\left(  0,T\right]  , \label{1.3}
\end{equation}
where $u\left(  t,x\right)  $ is the smooth solution of the following
anticipated parabolic partial integral differential equation
\[
\begin{array}
[c]{l}
\frac{\partial}{\partial t}\left(  t,x\right)  +\frac{1}{2}\sum\limits_{i=1}
^{d}\frac{\partial^{2}}{\partial x_{i}^{2}}u\left(  t,x\right)  +\int
_{\mathbb{R}}\frac{1}{\sqrt{2\pi}}f(t,u\left(  t,x\right)  ,\nabla_{x}u\left(
t,x\right)  ,\\
\text{ \ \ \ \ \ \ \ \ \ \ \ \ }u(t+\delta\left(  t\right)  ,x+\sqrt
{\delta\left(  t\right)  }y),\nabla_{x}u(t+\delta\left(  t\right)
,x+\sqrt{\delta\left(  t\right)  }y))e^{-\frac{\text{ \ }y^{2}}{2}}dy=0,
\end{array}
\]
with the terminal condition $u\left(  s,x\right)  =\varphi\left(  s,x\right)
$, $s\in\lbrack T,T+S]$, and $\nabla_{x}u$ is the gradient of $u$ with respect
to the spacial variable $x=\left(  x_{1},x_{2},\ldots,x_{d}\right)  $.
Furthermore, if $f\in C_{b}^{1+k,3+2k,3+2k,3+2k,3+2k}$, $\varphi\in
C_{b}^{1+k,4+2k}$ and $\sqrt{\delta(t)}\in C_{b}^{1+k}$, then we have $u\in$
$C_{b}^{1+k,3+2k}$, $k=0,1,2,....$

In this paper, we devote to proposing a class of stable explicit $\theta
$-schemes for solving anticipated BSDEs. The central idea is to use the
properties of the conditional mathematical expectation to subtly transform the
delay process in the generator into the current measurable process, which
obtain a high-order local truncation error. We also design effective numerical
methods to compute the conditional mathematical expectation with delay process
and get good results in our numerical tests. To the best of our knowledge,
this is the first attempt to come up with a high order numerical scheme for
solving anticipated BSDEs.

\bigskip For simple representations, we first introduce some notations that
will be used extensively throughout the paper:

\begin{itemize}
\item $\left\vert \cdot\right\vert :$ the standard Euclidean norm in
$\mathbb{R}$, $\mathbb{R}^{m}$ and $\mathbb{R}^{m\times d}$.

\item $L^{2}\left(  \mathcal{F}_{T};\mathbb{R}^{m}\right)  :\mathbb{R}^{m}
$-valued $\mathcal{F}_{T}$-measurable random variables such that
$\mathbb{E}[|\xi|^{2}]<\infty$.



\item $\mathbb{E}\left[  X\right]  :$ the conditional mathematical expectation
of $X$.

\item $\mathbb{E}_{t_{n}}^{x}[X]:$ the conditional mathematical expectation of
$X$, under the $\sigma$-field $\sigma\{W_{t}=x\}$, i.e. $\mathbb{E}\left[
X|W_{t_{n}}=x\right]  $.

\item $\Delta W_{t,s}:W_{s}-W_{t}$, $0\leq t\leq s,$ a standard Brownian
motion with mean zero and variance $s-t$.

\item $C_{b}^{l,k,k,k,k}:$ continuously differentiable functions $f:\left[
0,T\right]  \times\mathbb{R}^{m}\times\mathbb{R}^{m\times d}\times
\mathbb{R}^{m}\times\mathbb{R}^{m\times d}\rightarrow \mathbb{R}$ with
uniformly bounded partial derivatives $\partial_{t}^{l_{1}}f$ and
$\partial_{y_{1}}^{k_{1}}\partial_{z_{1}}^{k_{2}}\partial_{y_{2}}^{k_{3}
}\partial_{z_{2}}^{k_{4}}\\f$ for $\frac{1}{2}\leq l_{1}\leq l$ and $1\leq
\sum_{i=1}^{4}k_{i}\leq4$, $k_{i}\geq0$.

\item $C_{b}^{l,k}:$ continuously differentiable functions $\phi:\left[
0,T\right]  \times\mathbb{R}^{d}\rightarrow \mathbb{R}$ with uniformly bounded partial
derivatives $\partial_{t}^{l_{1}}\phi$ and $\partial_{x}^{k_{1}}\phi$ for
$\frac{1}{2}\leq l_{1}\leq l$ and $1\leq k_{1}\leq k$.

\item $C_{b}^{l}:$ continuously differentiable functions $\psi:\left[
0,T\right]  \rightarrow\mathbb{R}^{+}$ with uniformly bounded partial
derivatives $\partial_{t}^{l_{1}}\psi$ for $1\leq l_{1}\leq l$.
\end{itemize}

The paper is organized as follow. In section 2, we propose a class of explicit
$\theta$-schemes for anticipated BSDEs by choosing different parameters
$\theta_{i}(i=1,2,3)$. In section 3, according to the properties of the
conditional mathematical expectation, we subtly transform the delay process in
the generator into the current process, which obtain a high-order local
truncation error. We also analyze the stability of our numerical schemes and
strictly prove the error estimates. In section 4, various numerical tests\ are
given to demonstrate high efficiency and accuracy of our schemes. Finally, we
come to the conclusions in section 5.

\section{Explicit schemes for Anticipated BSDEs\label{sec3}}

\subsection{Reference equations\label{subsec3.1}}

For the time interval $\left[  0,T+S\right]  $, we introduce the following
uniform partition:
\[
0=t_{0}<t_{1}<\cdots<t_{M}=T+S,
\]
with $\Delta t=\frac{T+S}{M}$. For simplicity, we consider $S$ is a rational
number, by choosing $M$ appropriately, we let $S=K\Delta t$ and $T=N\Delta t$,
where $K$ and $N$ are positive integers.

Let $\left(  Y_{t,}Z_{t}\right)  $ be the solution of $\left(  \ref{1.2}
\right)  $, we consider the $\left(  \ref{1.2}\right)  $ on the time interval
$\left[  t_{n},t_{n+1}\right]  $, we can obtain easily
\begin{equation}
Y_{t_{n}}=Y_{t_{n+1}}+\int_{t_{n}}^{t_{n+1}}\mathbb{E}[\left.  f(s,Y_{s}
,Z_{s},Y_{s+\delta(s)},Z_{s+\delta(s)})\right\vert \mathcal{F}_{s}
]ds-\int_{t_{n}}^{t_{n+1}}Z_{s}dW_{s} \label{3.1}
\end{equation}
for $n=0,1,\ldots,N-1$. Taking the conditional mathematical expectation
$\mathbb{E}_{t_{n}}^{x}\left[  \cdot\right]  $ on both sides of $\left(
\ref{3.1}\right)  $, using the formula $\mathbb{E}_{t_{n}}^{x}[\mathbb{E}
\left[  \cdot|\mathcal{F}_{s}\right]  ]=\mathbb{E}_{t_{n}}^{x}[\cdot]$, for
$s\geq t_{n},$ we get
\begin{equation}
Y_{t_{n}}=\mathbb{E}_{t_{n}}^{x}[Y_{t_{n+1}}]+\int_{t_{n}}^{t_{n+1}}
\mathbb{E}_{t_{n}}^{x}[f(s,Y_{s},Z_{s},Y_{s+\delta(s)},Z_{s+\delta(s)})]ds.
\label{3.2}
\end{equation}
For notational simplicity, we denote $f_{s}=f(s,Y_{s},Z_{s},Y_{s+\delta
(s)},Z_{s+\delta(s)})$. It is notable that the integrands $\mathbb{E}_{t_{n}
}^{x}[f_{s}]$ in $\left(  \ref{3.2}\right)  $ is a deterministic function,
then we can use some numerical integration methods to approximate it. Here we
use the $\theta$-scheme to approximate the integral:
\begin{equation}
\int_{t_{n}}^{t_{n+1}}\mathbb{E}_{t_{n}}^{x}[f_{s}]ds=\theta_{1}\Delta
t\mathbb{E}_{t_{n}}^{x}[f_{t_{n}}]+\left(  1-\theta_{1}\right)  \Delta
t\mathbb{E}_{t_{n}}^{x}[f_{t_{n+1}}]+R_{y_{1}}^{n,\delta}, \label{3.3}
\end{equation}
where $\theta_{1}\in\left[  0,1\right]  $ and
\begin{equation}
R_{y_{1}}^{n,\delta}=\int_{t_{n}}^{t_{n+1}}\{\mathbb{E}_{t_{n}}^{x}
[f_{s}]-\theta_{1}\mathbb{E}_{t_{n}}^{x}[f_{t_{n}}]+(1-\theta_{1}
)\mathbb{E}_{t_{n}}^{x}[f_{t_{n+1}}]\}ds. \label{3.4}
\end{equation}
Since $t_{n}+\delta(t_{n})$ and $t_{n+1}+\delta(t_{n+1})$ may not be on the
grid points, we approximate it with the value of the adjacent grid points
$t_{n+\delta_{n}^{-}}$, $t_{n+\delta_{n}^{+}}$, $t_{n+1+\delta_{n+1}^{-}}$and
$t_{n+1+\delta_{n+1}^{+}}$, then
\begin{equation}
\begin{array}
[c]{c}
\theta_{1}\Delta t\mathbb{E}_{t_{n}}^{x}\left[  f_{t_{n}}\right]  =\theta
_{1}\Delta t\mathbb{E}_{t_{n}}^{x}\left[  k_{n}f_{t_{n}^{-}}+\left(
1-k_{n}\right)  f_{t_{n}^{+}}\right]  +R_{y_{2,1}}^{n,\delta},\\
\text{ }(1-\theta_{1})\Delta t\mathbb{E}_{t_{n}}^{x}\left[  f_{t_{n+1}
}\right]  =(1-\theta_{1})\Delta t\mathbb{E}_{t_{n}}^{x}\left[  k_{n+1}
f_{t_{n+1}^{-}}+\left(  1-k_{n+1}\right)  f_{t_{n+1}^{+}}\right]  +R_{y_{2,2}
}^{n,\delta},
\end{array}
\label{3.5}
\end{equation}
where
\[
\begin{array}
[c]{l}
\text{ \ }f_{t_{n}^{-}}=f(t_{n},Y_{t_{n}},Z_{t_{n}},Y_{t_{n+\delta_{n}^{-}}
},Z_{t_{n+\delta_{n}^{-}}}),\\
\text{ \ }f_{t_{n}^{+}}=f(t_{n},Y_{t_{n}},Z_{t_{n}},Y_{t_{n+\delta_{n}^{+}}
},Z_{t_{n+\delta_{n}^{+}}}),\\
\text{ \ }f_{t_{n+1}^{+}}=f(t_{n+1},Y_{t_{n+1}},Z_{t_{n+1}},Y_{t_{n+1+\delta
_{n+1}^{-}}},Z_{t_{n+1+\delta_{n+1}^{-}}}),\\
\text{ \ }f_{t_{n+1}^{+}}=f(t_{n+1},Y_{t_{n+1}},Z_{t_{n+1}},Y_{t_{n+1+\delta
_{n+1}^{+}}},Z_{t_{n+1+\delta_{n+1}^{+}}}),
\end{array}
\]
and
\begin{equation}
\begin{array}
[c]{ll}
\delta_{n}^{-}=[\frac{\delta\left(  t_{n}\right)  }{\Delta t}], & \delta
_{n}^{+}=[\frac{\delta\left(  t_{n}\right)  }{\Delta t}]+1,\\
k_{n}=\frac{\delta_{n}^{+}\Delta t-\delta\left(  t_{n}\right)  }{\Delta t}, &
k_{n+1}=\frac{\delta_{n+1}^{+}\Delta t-\delta\left(  t_{n+1}\right)  }{\Delta
t},
\end{array}
\label{Kn}
\end{equation}
the scale coefficients $k_{n}\in\lbrack0,1]$, $n=0,1,\ldots,N-1$. Combining
$\left(  \ref{3.3}\right)  $ and $\left(  \ref{3.5}\right)  ,$ we deduce that
\begin{equation}
\begin{array}
[c]{l}
Y_{t_{n}}=\mathbb{E}_{t_{n}}^{x}[Y_{t_{n+1}}]+\theta_{1}k_{n}\Delta
t\mathbb{E}_{t_{n}}^{x}[f_{t_{n}^{-}}]+\theta_{1}\left(  1-k_{n}\right)
\Delta t\mathbb{E}_{t_{n}}^{x}[f_{t_{n}^{+}}]+R_{y_{1}}^{n,\delta}\\
\text{ \ \ \ \ \ \ }+(1-\theta_{1})k_{n+1}\Delta t\mathbb{E}_{t_{n}}
^{x}[f_{t_{n+1}^{-}}]+(1-\theta_{1})\left(  1-k_{n+1}\right)  \Delta
t\mathbb{E}_{t_{n}}^{x}[f_{t_{n+1}^{+}}]+R_{y_{2}}^{n,\delta},
\end{array}
\label{3.6}
\end{equation}
where interpolation error $R_{y_{2}}^{n,\delta}=R_{y_{2,1}}^{n,\delta
}+R_{y_{2,2}}^{n,\delta}$. In order to propose an explicit numerical scheme to
solve $\left\{  Y_{t_{n}}\right\}  $, we approximate the $Y_{t_{n}}$ in
$f_{t_{n}^{-}}$ and \ $f_{t_{n}^{+}}$ by the left rectangle formula
\begin{equation}
Y_{t_{n}}=\mathbb{E}_{t_{n}}^{x}[Y_{t_{n+1}}]+\Delta t\mathbb{E}_{t_{n}}
^{x}[f_{t_{n+1}^{-}}]+\tilde{R}_{y}^{n,\delta}, \label{3.7}
\end{equation}
where $\tilde{R}_{y}^{n}=\int_{t_{n}}^{t_{n+1}}\mathbb{E}_{t_{n}}^{x}\left[
f_{s}\right]  -\mathbb{E}_{t_{n}}^{x}[f_{t_{n+1}^{-}}]ds$. According to $\left(\ref{3.6}\right)  $
and $\left(  \ref{3.7}\right)  $, we obtain the following
reference equation:
\begin{equation}
\begin{array}
[c]{l}
Y_{t_{n}}=\mathbb{E}_{t_{n}}^{x}[Y_{t_{n+1}}]+\theta_{1}k_{n}\Delta
t\mathbb{E}_{t_{n}}^{x}[\bar{f}_{t_{n}^{-}}]+\theta_{1}(1-k_{n})\Delta
t\mathbb{E}_{t_{n}}^{x}[\bar{f}_{t_{n}^{+}}]\\
\text{ \ \ \ \ \ \ }+(1-\theta_{1})k_{n+1}\Delta t\mathbb{E}_{t_{n}}
^{x}[f_{t_{n+1}^{-}}]+(1-\theta_{1})\left(  1-k_{n+1}\right)  \Delta
t\mathbb{E}_{t_{n}}^{x}[f_{t_{n+1}^{+}}]+R_{y}^{n,\delta},
\end{array}
\label{3.8ref}
\end{equation}
where
\[
\begin{array}
[c]{l}
\bar{f}_{t_{n}^{-}}=f(t_{n},\mathbb{E}_{t_{n}}^{x}[Y_{t_{n+1}}]+\Delta
t\mathbb{E}_{t_{n}}^{x}[f_{t_{n+1}^{-}}],Z_{t_{n}},Y_{t_{n+\delta_{n}^{-}}
},Z_{t_{n+\delta_{n}^{-}}}),\\
\bar{f}_{t_{n}^{+}}=f(t_{n},\mathbb{E}_{t_{n}}^{x}[Y_{t_{n+1}}]+\Delta
t\mathbb{E}_{t_{n}}^{x}[f_{t_{n+1}^{-}}],Z_{t_{n}},Y_{t_{n+\delta_{n}^{+}}
},Z_{t_{n+\delta_{n}^{+}}}),
\end{array}
\]
and $R_{y}^{n,\delta}=R_{y_{1}}^{n,\delta}+R_{y_{2}}^{n,\delta}+R_{y_{3}
}^{n,\delta}$ with
\begin{equation}
R_{y_{3}}^{n,\delta}=\theta_{1}k_{n}\Delta t\mathbb{E}_{t_{n}}^{x}[\bar
{f}_{t_{n}^{-}}-f_{t_{n}^{-}}]+\theta_{1}\left(  1-k_{n}\right)  \Delta
t\mathbb{E}_{t_{n}}^{x}[\bar{f}_{t_{n}^{+}}-f_{t_{n}^{+}}].
\end{equation}

Now, multiplying $\left(  \ref{3.1}\right)$ by $\Delta W_{t_{n},t_{n+1}}^{\intercal}$
and taking the conditional mathematical expectation
$\mathbb{E}_{t_{n}}^{x}\left[  \cdot\right]  $ on both sides of the derived
equation, by the It\^{o} isometry formula, we obtain
\begin{equation}\label{3.9}
0=\mathbb{E}_{t_{n}}^{x}[Y_{t_{n+1}}\Delta W_{t_{n},t_{n+1}}^{\intercal}]
+\int_{t_{n}}^{t_{n+1}}\mathbb{E}_{t_{n}}^{x}[f_{s}\Delta W_{t_{n},s}^{\intercal}]ds-
\int_{t_{n}}^{t_{n+1}}\mathbb{E}_{t_{n}}^{x}[Z_{s}]ds.
\end{equation}
To simplify the presentation, we use $\Delta W_{t_{n+1}}^{\intercal}$to denote
$\Delta W_{t_{n},t_{n+1}}^{\intercal}$. Similar to the above process, we
rewrite the two integral terms on the right side of $\left(  \ref{3.9}\right)
$ in the following:

\begin{equation}
\begin{array}
[c]{l}
\int_{t_{n}}^{t_{n+1}}\mathbb{E}_{t_{n}}^{x}[f_{s}\Delta W_{t_{n}
,s}^{\intercal}]ds=(1-\theta_{2})k_{n+1}\Delta t\mathbb{E}_{t_{n}}
^{x}[f_{t_{n+1}^{-}}\Delta W_{t_{n+1}}^{\intercal}]+R_{z_{1}}^{n,\delta}\\
\text{ \ \ \ \ \ \ \ \ \ \ \ \ \ \ \ \ \ \ \ \ \ \ \ \ \ \ \ \ \ \ \ }
+(1-\theta_{2})\left(  1-k_{n+1}\right)  \Delta t\mathbb{E}_{t_{n}}
^{x}[f_{t_{n+1}^{+}}\Delta W_{t_{n+1}}^{\intercal}]+R_{z_{2}}^{n,\delta},
\end{array}
\label{3.10}
\end{equation}
\begin{equation}
-\int_{t_{n}}^{t_{n+1}}\mathbb{E}_{t_{n}}^{x}[Z_{s}]ds=-\theta_{3}\Delta
tZ_{t_{n}}-\left(  1-\theta_{3}\right)  \Delta t\mathbb{E}_{t_{n}}
^{x}[Z_{t_{n+1}}]+R_{z_{3}}^{n,\delta}, \label{3.11}
\end{equation}
where $\theta_{2}\in\left[  0,1\right]  $, $\theta_{3}\in(0,1]$

\begin{align}
R_{z_{1}}^{n,\delta}  &  =\int_{t_{n}}^{t_{n+1}}\mathbb{E}_{t_{n}}^{x}
[f_{s}\Delta W_{t_{n},s}^{\intercal}]ds-\left(  1-\theta_{2}\right)
\mathbb{E}_{t_{n}}^{x}[f_{t_{n+1}}\Delta W_{t_{n+1}}^{\intercal}]\Delta t,\\
R_{z_{2}}^{n,\delta}  &  =(1-\theta_{2})\Delta t\mathbb{E}_{t_{n}}^{x}\left[
\left(  f_{t_{n+1}}-k_{n+1}f_{t_{n+1}^{-}}-\left(  1-k_{n+1}\right)
f_{t_{n+1}^{+}}\right)  \Delta W_{t_{n+1}}^{\intercal}\right]  ,
\end{align}
and
\begin{equation}
\begin{array}
[c]{l}
R_{z_{3}}^{n,\delta}=-\int_{t_{n}}^{t_{n+1}}\left\{  \mathbb{E}_{t_{n}}
^{x}[Z_{s}]-\theta_{3}Z_{t_{n}}-(1-\theta_{3})\mathbb{E}_{t_{n}}
^{x}[Z_{t_{n+1}}]\right\}  ds.\text{ \ \ \ \ \ \ \ \ \ \ \ \ \ \ \ }
\end{array}
\end{equation}
According to $\left(  \ref{3.9}\right)  ,\left(  \ref{3.10}\right)  $ and
$\left(  \ref{3.11}\right)  $, we obtain
\begin{equation}
\begin{array}
[c]{l}
\theta_{3}\Delta tZ_{t_{n}}=\mathbb{E}_{t_{n}}^{x}[Y_{t_{n+1}}\Delta
W_{t_{n+1}}^{\intercal}]+(1-\theta_{2})\left(  1-k_{n+1}\right)  \Delta
t\mathbb{E}_{t_{n}}^{x}[f_{t_{n+1}^{+}}\Delta W_{t_{n+1}}^{\intercal}]\\
\text{\ \ \ \ \ \ \ \ \ \ \ \ \ }+(1-\theta_{2})k_{n+1}\Delta t\mathbb{E}
_{t_{n}}^{x}[f_{t_{n+1}^{-}}\Delta W_{t_{n+1}}^{\intercal}]-(1-\theta
_{3})\Delta t\mathbb{E}_{t_{n}}^{x}[Z_{t_{n+1}}]+R_{z}^{n,\delta},
\end{array}
\label{3.12}
\end{equation}
where $R_{z}^{n,\delta}=R_{z_{1}}^{n,\delta}+R_{z_{2}}^{n,\delta}+R_{z_{3}
}^{n,\delta}$. By $\left(  \ref{1.3}\right)  $ and the definition of
$\mathbb{E}_{t_{n}}^{x}\left[  \cdot\right]  $, using the formula of
integration by part, it is easy to prove that:
\begin{equation}
\mathbb{E}_{t_{n}}^{x}[Y_{t_{n+1}}\Delta W_{t_{n+1}}^{\intercal}]=\Delta
t\mathbb{E}_{t_{n}}^{x}[\nabla Y_{t_{n+1}}]=\Delta t\mathbb{E}_{t_{n}}
^{x}[Z_{t_{n+1}}]. \label{3.13}
\end{equation}
Substituting $\left(  \ref{3.13}\right)  $ into $\left(  \ref{3.12}\right)  $,
we get the following reference equation:
\begin{equation}
\begin{array}
[c]{l}
\theta_{3}\Delta tZ_{t_{n}}=\theta_{3}\mathbb{E}_{t_{n}}^{x}[Y_{t_{n+1}}\Delta
W_{t_{n+1}}^{\intercal}]+(1-\theta_{2})k_{n+1}\Delta t\mathbb{E}_{t_{n}}
^{x}[f_{t_{n+1}^{-}}\Delta W_{t_{n+1}}^{\intercal}]\\
\text{ \ \ \ \ \ \ \ \ \ \ \ \ \ }+(1-\theta_{2})\left(  1-k_{n+1}\right)
\Delta t\mathbb{E}_{t_{n}}^{x}[f_{t_{n+1}^{+}}\Delta W_{t_{n+1}}^{\intercal
}]+R_{z}^{n,\delta}.
\end{array}
\label{3.14ref}
\end{equation}

\subsection{The explicit $\theta$-schemes\label{subsec3.2}}

Let $\left(  Y^{n},Z^{n}\right)  $ $(n=N-1,\ldots,0)$ denote the numerical
approximation to the solution $\left(  Y_{t_{n}},Z_{t_{n}}\right)  $ of the
anticipated BSDEs $\left(  \ref{1.2}\right)  $ at time level $t=t_{n}
(n=N-1,\ldots,0)$. To maintain the consistency of representation, we use
$\left(  Y^{N+i},Z^{N+i}\right)  $ to denote the given final values $\left(
Y_{t_{n}},Z_{t_{n}}\right)  $ of the anticipated BSDEs $\left(  \ref{1.2}
\right)  $ at time partition $t=t_{N+i}(i=0,\ldots,K)$. From the two reference
equations $\left(  \ref{3.8ref}\right)  $ and $\left(  \ref{3.14ref}\right)
$, we propose the following explicit $\theta$-schemes for solving the
anticipated BSDEs $\left(  \ref{1.2}\right)  $.

\textbf{The explicit }$\theta$-\textbf{schemes\label{scheme} }Given the random
variables $Y^{N+i}$ and $Z^{N+i}$, $i=0,\ldots,K$, solve random variables
$Y^{n}$ and $Z^{n}$, $n=N-1,\ldots,0$, from
\begin{align}
\theta_{3}\Delta tZ^{n}  &  =\theta_{3}\mathbb{E}_{t_{n}}^{x}[Y^{n+1}\Delta
W_{t_{n+1}}^{\intercal}]+(1-\theta_{2})k_{n+1}\Delta t\mathbb{E}_{t_{n}}
^{x}[f_{-}^{n+1}\Delta W_{t_{n+1}}^{\intercal}]\nonumber\\
&  \text{ \ \ }+(1-\theta_{2})\left(  1-k_{n+1}\right)  \Delta t\mathbb{E}
_{t_{n}}^{x}[f_{+}^{n+1}\Delta W_{t_{n+1}}^{\intercal}], \label{scheme3.2.1}
\end{align}
\begin{equation}
\left\{
\begin{array}
[c]{l}
\bar{Y}^{n}=\mathbb{E}_{t_{n}}^{x}[Y^{n+1}]+\Delta t\mathbb{E}_{t_{n}}
^{x}[f_{-}^{n+1}],\\
Y^{n}=\mathbb{E}_{t_{n}}^{x}[Y^{n+1}]+\theta_{1}k_{n}\Delta t\mathbb{E}
_{t_{n}}^{x}\left[  \bar{f}_{-}^{n}\right]  +\theta_{1}(1-k_{n})\Delta
t\mathbb{E}_{t_{n}}^{x}\left[  \bar{f}_{+}^{n}\right] \\
\text{ \ \ \ \ \ \ }+(1-\theta_{1})k_{n+1}\Delta t\mathbb{E}_{t_{n}}^{x}
[f_{-}^{n+1}]+(1-\theta_{1})(1-k_{n+1})\Delta t\mathbb{E}_{t_{n}}^{x}
[f_{+}^{n+1}],
\end{array}
\right.  \label{scheme3.2.2}
\end{equation}
with the deterministic parameters $\theta_{i}\in\left[  0,1\right]  (i=1,2)$,
$\theta_{3}\in(0,1]$, the scale coefficients $k_{n}\in\left[  0,1\right]
(n=0,1,\ldots,N-1)$ defined in $\left(  \ref{Kn}\right)  $ and
\begin{equation}
\begin{array}
[c]{l}
\bar{f}_{-}^{n}=f(t_{n},\bar{Y}^{n},Z^{n},Y^{n+\delta_{n}^{-}},Z^{n+\delta
_{n}^{-}}),\\
\bar{f}_{+}^{n}=f(t_{n},\bar{Y}^{n},Z^{n},Y^{n+\delta_{n}^{+}},Z^{n+\delta
_{n}^{+}}),\\
f_{-}^{n+1}=f(t_{n+1},Y^{n+1},Z^{n+1},Y^{n+1+\delta_{n+1}^{-}},Z^{n+1+\delta
_{n+1}^{-}}),\\
f_{+}^{n+1}=f(t_{n+1},Y^{n+1},Z^{n+1},Y^{n+1+\delta_{n+1}^{+}},Z^{n+1+\delta
_{n+1}^{+}}).
\end{array}
\end{equation}

\begin{remark}
It can be seen that by choosing different parameters $\theta_{i}\left(
i=1,2\right)  \in\lbrack0,1]$, $\theta_{3}\in(0,1]$, we could get a class of
different numerical schemes. When choosing $\theta_{i}=\frac{1}{2}(i=1,2,3)$,
the explicit $\theta$-scheme is a second order scheme for solving anticipated
BSDEs, which will be confirmed in later numerical experiments.
\end{remark}

\begin{remark}\label{remark}
When we calculate the $Y^{n}$ and $Z^{n}$ at each time level, the future
values are also involved. By reasonable selecting $M$, the future values can
be on the time grid points to ensure the performed of the numerical
iterations. In addition, when $Y^{n}$ and $Z^{n}$ are near the time of 0,
the point of $\delta_{n}^{-}=0$ may appear,
we take $\bar{f}_{-}^{n}=f(t_{n},\bar{Y}^{n},Z^{n},\bar{Y}^{n},Z^{n})$ at
this time.
\end{remark}

\section{Error estimates \label{sec4}}

In this section, we will give the error estimates of the explicit $\theta
$-schemes $\left(  \ref{scheme3.2.1}\right)  $ and $\left(  \ref{scheme3.2.2}
\right)  $. Before given the main theorem, we first introduce some useful
lemmas which play an important role in the later proof.

\subsection{Stability analysis\label{subsec4.1}}

\begin{lemma}
[Discrete Gronwall Lemma]\label{gronwall lemma}Suppose that $N$ and $K$ are
two nonnegative integers and $\Delta t$ is any positive number. Let $\left\{
A_{n}\right\}  $, $n=N-1,N-2,\ldots,0$, satisfy
\[
\left\vert A_{n}\right\vert \leq\beta+\alpha\Delta t\sum_{i=n+1}
^{N+K}\left\vert A_{i}\right\vert ,
\]
where $\alpha$ and $\beta$ are two positive constants. Let $M_{A}
=\max\limits_{N\leq j\leq N+K}\left\vert A_{j}\right\vert $ and $\hat
{T}=T+S=\left(  N+K\right)  \Delta t$, then for $n=N-1,N-2,\ldots,0$,
\[
\left\vert A_{n}\right\vert \leq e^{\alpha\hat{T}}\left(  \beta+\alpha K\Delta
tM_{A}\right)  .
\]

\end{lemma}

\begin{lemma}
\label{lemma}Let $N$ and $K$ are two nonnegative integers, $\Delta t$ is any
positive number. Assume $\left\{  A_{n}\right\}  ,\left\{  B_{n}\right\}  $
and $\left\{  R^{n}\right\}  ,n=0,1,\ldots,N+K$, are nonnegative sequences,
satisfy the inequality
\begin{equation}
A_{n}+C_{1}\Delta tB_{n}\leq\left(  1+C_{2}\Delta t\right)  A_{n+1}
+C_{3}\Delta t\sum_{i\in\mathcal{I}_{n}}\left(  A_{n+i}+B_{n+i}\right)
+R^{n}, \label{4.1}
\end{equation}
for $n=0,1,\ldots,N-1$, where $\mathcal{I}_{n}$ $\subseteq$ $\left\{
1,2,\ldots,K+1\right\}  $ is an indicator set with $\left\vert \mathcal{I}
_{n}\right\vert \leq I_{0}<\infty$, $C_{1},C_{2}$ and $C_{3}$ are positive
constants. Let
\[
M_{A}=\max_{N\leq i\leq N+K}A_{i},\;\;\;M_{B}=\max_{N\leq i\leq N+K}
B_{i},\;\;\;\hat{T}=\left(  N+K\right)  \Delta t.
\]
If there exist a constant $C_{4}$ $>0$, such that $C_{1}-I_{0}C_{3}\geq C_{4}
$, Then for $n=N-1,N-2,\ldots,0$, it holds that
\begin{equation}
A_{n}+\Delta t\sum_{i=n}^{N-1}B_{i}\leq C\left[  A_{N}+M_{A}+M_{B}+\sum
_{i=0}^{N-1}R^{i}\right]  , \label{4.2}
\end{equation}
where $C$ is a constant depending on $C_{1},C_{2},C_{3},\hat{T}$ and $I_{0}$.
\end{lemma}

\begin{proof}
This proof is similar to the Lemma 4.2 in \cite{ZhZhJ2016}. See the appendix
for details.
\end{proof}

\begin{remark}
The Lemma \ref{lemma} is a generalization of the Lemma 4.2 in \cite{ZhZhJ2016}
. When $\Delta t\rightarrow0$, then $K\rightarrow\infty$ and $\mathcal{I}_{n}$
$\subseteq$ $\left\{  1,2,\ldots,K+1\right\}  $ can include the infinity
number indicator, resulting in the right term of the $\left(  \ref{4.2}
\right)  $ changing from $\Delta t(M_{A}+M_{B})$ to $(M_{A}+M_{B})$.
\end{remark}

Without loss of generality, in the following of this section we only consider
the case of one-dimensional anticipated BSDEs $(i.e.,m=d=1)$. Conclusions can
be generalized to $d$-dimensional anticipated BSDEs where $d>1$. Before the
main theorem, we first introduce the following notation:
\[
\begin{array}
[c]{ll}
e_{y}^{n}=Y_{t_{n}}-Y^{n}, & e_{z}^{n}=Z_{t_{n}}-Z^{n},\\
e_{\bar{f}_{-}}^{n}=\bar{f}_{t_{n}^{-}}-\bar{f}_{-}^{n}, & e_{\bar{f}_{+}}
^{n}=\bar{f}_{t_{n}^{+}}-\bar{f}_{+}^{n},\\
e_{f_{-}}^{n+1}=f_{t_{n+1}^{-}}-f_{-}^{n+1}, & e_{f_{+}}^{n+1}=f_{t_{n+1}^{+}
}-f_{+}^{n+1},
\end{array}
\]
and
\[
\begin{array}
[c]{ll}
e_{y}^{n+\delta_{n}^{-}}=Y_{t_{n+\delta_{n}^{-}}}-Y^{n+\delta
_{n}^{-}}, & e_{z}^{n+\delta_{n}^{-}}=Z_{t_{n+\delta_{n}^{-}}
}-Z^{n+\delta_{n}^{-}},\\
e_{y}^{n+\delta_{n}^{+}}=Y_{t_{n+\delta_{n}^{+}}}-Y^{n+\delta
_{n}^{+}}, & e_{z}^{n+\delta_{n}^{+}}=Z_{t_{n+\delta_{n}^{+}}
}-Z^{n+\delta_{n}^{+}},\\
e_{y}^{n+1+\delta_{n+1}^{-}}=Y_{t_{n+1+\delta_{n+1}^{-}}
}-Y^{n+1+\delta_{n+1}^{-}}, & e_{z}^{n+1+\delta_{n+1}^{-}
}=Z_{t_{n+1+\delta_{n+1}^{-}}}-Z^{n+1+\delta_{n+1}^{-}},\\
e_{y}^{n+1+\delta_{n+1}^{+}}=Y_{t_{n+1+\delta_{n+1}^{+}}
}-Y^{n+1+\delta_{n+1}^{+}}, & e_{z}^{n+1+\delta_{n+1}^{+}
}=Z_{t_{n+1+\delta_{n+1}^{+}}}-Z^{n+1+\delta_{n+1}^{+}}.
\end{array}
\]

\begin{theorem}
\label{Theorem 4.1} Let $\left(  Y_{t},Z_{t}\right)  $, $t\in\left[
0,T+S\right]  $, be the solution of the anticipated BSDEs $\left(
\ref{1.2}\right)  $ and $\left(  Y^{n},Z^{n}\right)(n=N+K,\ldots
,0)  $ defined in the explicit $\theta$-schemes \ref{scheme},
$M_{ey}=\max\limits_{N\leq i\leq N+K}\mathbb{E}[|e_{y}^{i}|^{2}],\;M_{ez}
=\max\limits_{N\leq i\leq N+K}\mathbb{E}[|e_{z}^{i}|^{2}],\;\hat{T}=\left(
N+K\right)  \Delta t=T+S$. Assume that $f$ satisfies the regularity conditions
with Lipschitz constant $L$. Then for sufficiently small time step $\Delta t$,
we have
\begin{equation}
\begin{array}
[c]{l}
\mathbb{E[}|e_{y}^{n}|^{2}]+\Delta t\sum\limits_{i=n}^{N-1}\mathbb{E[}
\left\vert e_{z}^{n}\right\vert ^{2}]\\
\;\ \ \ \ \ \ \ \ \ \text{\ \ }\leq C\left[  \mathbb{E[}|e_{y}^{N}
|^{2}]+M_{ey}+M_{ez}+\sum\limits_{i=0}^{N-1}\frac{\left(  1+\Delta t\right)
\mathbb{E[|}R_{y}^{n,\delta}|^{2}]+\mathbb{E[|}R_{z}^{n,\delta}|^{2}]}{\Delta
t}\right]  ,
\end{array}
\ \label{theorem4.1}
\end{equation}
for $n=N-1,N-2,\ldots,0$. Here $C$ is a constant depending on $L$ and $\hat
{T}$, $R_{y}^{n,\delta}\mathbb{\ }$and $R_{z}^{n,\delta}$are defined in
$\left(  \ref{3.8ref}\right)  $ and $\left(  \ref{3.14ref}\right)  $, respectively.

\begin{proof}
We complete the proof of the theorem in three steps.

Step1:The estimate of $e_{y}^{n}$.

Since $\delta(t)$ is a monotonically increasing function, we denote the first
$n_{0}\ $which satisfies $\delta_{n}^{-}=1$. While $n\,\geq n_{0}$, then $\delta_{n}
^{-}\geq1$, for each integer $n>0$, Subtract $\left(  \ref{3.8ref}\right)  $ from $\left(
\ref{scheme3.2.2}\right)  $, we can obtain:
\begin{equation}
\begin{array}
[c]{l}
e_{y}^{n}=\mathbb{E}_{t_{n}}^{x}[e_{y}^{n+1}]+\theta_{1}k_{n}\Delta
t\mathbb{E}_{t_{n}}^{x}[e_{\bar{f}_{-}}^{n}]+\theta_{1}(1-k_{n})\Delta
t\mathbb{E}_{t_{n}}^{x}[e_{\bar{f}_{+}}^{n}]\\
\text{ \ \ \ \ }+\left(  1-\theta_{1}\right)  k_{n+1}\Delta t\mathbb{E}
_{t_{n}}^{x}[e_{f_{-}}^{n+1}]+\left(  1-\theta_{1}\right)  (1-k_{n+1})\Delta
t\mathbb{E}_{t_{n}}^{x}[e_{f_{+}}^{n+1}]+R_{y}^{n,\delta}.
\end{array}
\label{4.11}
\end{equation}
Then
\begin{equation}
\begin{array}
[c]{l}
|e_{y}^{n}|\leq|\mathbb{E}_{t_{n}}^{x}[e_{y}^{n+1}]|+\theta_{1}k_{n}\Delta
t\mathbb{E}_{t_{n}}^{x}[|e_{\bar{f}_{-}}^{n}|]+\theta_{1}(1-k_{n})\Delta
t\mathbb{E}_{t_{n}}^{x}[|e_{\bar{f}_{+}}^{n}|]\\
\text{ \ \ \ \ \ \ \ }+\left(  1-\theta_{1}\right)  k_{n+1}\Delta
t\mathbb{E}_{t_{n}}^{x}[|e_{f_{-}}^{n+1}|]+\left(  1-\theta_{1}\right)
(1-k_{n+1})\Delta t\mathbb{E}_{t_{n}}^{x}[|e_{f_{+}}^{n+1}|]+|R_{y}^{n,\delta
}|.
\end{array}
\label{4.12}
\end{equation}
By the definition of $e_{\bar{f}_{-}}^{n},$ $e_{\bar{f}_{+}}^{n},$ $e_{f_{-}
}^{n+1}$and $e_{f_{+}}^{n+1}$, using the Lipschitz continuity property of $f$,
we have
\begin{equation}
\begin{array}
[c]{l}
\mathbb{E}_{t_{n}}^{x}[|e_{\bar{f}_{-}}^{n}|]\leq L\left\{ \Delta tL(\mathbb{E}_{t_{n}}^{x}[|e_{y}^{n+1}|
+|e_{z}^{n+1}|+
|e_{y}^{n+1+\delta_{n+1}^{-}}|+|e_{z}^{n+1+\delta_{n+1}^{-}}|])\right. \\
\text{\ \ \ \ \ \ \ \ \ \ \ \ \ \ }\left.   +\mathbb{E}_{t_{n}
}^{x}[|e_{y}^{n+1}|]+|e_{z}^{n}|+\mathbb{E}_{t_{n}
}^{x}[|e_{y}^{n+\delta_{n}^{-}}|+|e_{z}^{n+\delta_{n}^{-}}|]\right\} \\
\text{ \ \ \ \ \ \ \ \ \ \ \ }\leq L\left\{  (1+\Delta tL)(\mathbb{E}_{t_{n}
}^{x}[|e_{y}^{n+1}|+|e_{z}^{n+1}|+|e_{y}^{n+1+\delta_{n+1}^{-}}|+|e_{z}
^{n+1+\delta_{n+1}^{-}}|])\right. \\
\text{\ \ \ \ \ \ \ \ \ \ \ \ \ \ }\left.  +|e_{z}^{n}|+\mathbb{E}_{t_{n}}^{x}[|e_{y}^{n+\delta_{n}^{-}}|+|e_{z}^{n+\delta
_{n}^{-}}|]\right\},
\end{array}
\label{4.13}
\end{equation}
and
\begin{equation}
\mathbb{E}_{t_{n}}^{x}[|e_{f_{-}}^{n+1}|]\leq L\left(  \mathbb{E}_{t_{n}}
^{x}[|e_{y}^{n+1}|+|e_{z}^{n+1}|+|e_{y}^{n+1+\delta_{n+1}^{-}}|+|e_{z}
^{n+1+\delta_{n+1}^{-}}|]\right)  . \label{4.14}
\end{equation}
Similarly
\[
\begin{array}
[c]{l}
\mathbb{E}_{t_{n}}^{x}[|e_{\bar{f}_{+}}^{n}|]\leq L\left\{  (1+\Delta
tL)\mathbb{E}_{t_{n}}^{x}[|e_{y}^{n+1}|+|e_{z}^{n+1}|+|e_{y}^{n+1+\delta
_{n+1}^{-}}|+|e_{z}^{n+1+\delta_{n+1}^{-}}|]\right. \\
\text{\ \ \ \ \ \ \ \ \ \ \ \ \ \ }\left.  +|e_{z}^{n}|+\mathbb{E}_{t_{n}}^{x}[|e_{y}^{n+\delta_{n}^{+}}|+|e_{z}^{n+\delta
_{n}^{+}}|]\right\}  ,
\end{array}
\]
and
\[
\mathbb{E}_{t_{n}}^{x}[|e_{f_{+}}^{n+1}|]\leq L\left(  \mathbb{E}_{t_{n}}
^{x}[  |e_{y}^{n+1}|+|e_{z}^{n+1}|+|e_{y}^{n+1+\delta_{n+1}^{+}}
|+|e_{z}^{n+1+\delta_{n+1}^{+}}|]  \right)  .
\]
We may rewrite
\begin{equation}
\begin{array}
[c]{l}
|e_{y}^{n}|\leq\left\vert \mathbb{E}_{t_{n}}^{x}[e_{y}^{n+1}]\right\vert
+\theta_{1}L\Delta t(1+L\Delta t)\mathbb{E}_{t_{n}}^{x}\left[  |e_{y}
^{n+1}|+|e_{z}^{n+1}|\right] \\
\text{ \ \ \ \ \ \ \ }+\theta_{1}L\Delta t(1+L\Delta t)\mathbb{E}_{t_{n}}
^{x}[|e_{y}^{n+1+\delta_{n+1}^{-}}|+|e_{z}^{n+1+\delta_{n+1}^{-}}|]\\
\text{ \ \ \ \ \ \ \ }+\theta_{1}L\Delta t\mathbb{E}_{t_{n}}^{x}
[|e_{y}^{n+\delta_{n}^{-}}|+|e_{z}^{n+\delta_{n}^{-}}|+|e_{y}^{n+\delta
_{n}^{+}}|+|e_{z}^{n+\delta_{n}^{+}}|]\\
\text{ \ \ \ \ \ \ \ }+(1-\theta_{1})L\Delta t\mathbb{E}_{t_{n}}^{x}
[|e_{y}^{n+1}|+|e_{z}^{n+1}|]+\theta_{1}L\Delta t|e_{z}^{n}|\\
\text{ \ \ \ \ \ \ \ }+(1-\theta_{1})L\Delta t\mathbb{E}_{t_{n}}^{x}
[|e_{y}^{n+1+\delta_{n+1}^{-}}|+|e_{z}^{n+1+\delta_{n+1}^{-}}|]\\
\text{ \ \ \ \ \ \ \ }+(1-\theta_{1})L\Delta t\mathbb{E}_{t_{n}}^{x}
[|e_{y}^{n+1+\delta_{n+1}^{+}}|+|e_{z}^{n+1+\delta_{n+1}^{+}}|]+|R_{y}
^{n,\delta}|\\
\text{ \ \ \ }\leq\left\vert \mathbb{E}_{t_{n}}^{x}[e_{y}^{n+1}]\right\vert
+L\left(  1+L\right)  \Delta t\mathbb{E}_{t_{n}}^{x}[|e_{y}^{n+1}
|+|e_{z}^{n+1}|]\\
\text{ \ \ \ \ \ \ \ }+L(1+L)\Delta t\mathbb{E}_{t_{n}}^{x}[|e_{y}
^{n+1+\delta_{n+1}^{-}}|+|e_{z}^{n+1+\delta_{n+1}^{-}}|]\\
\text{ \ \ \ \ \ \ \ }+L(1+L)\Delta t\mathbb{E}_{t_{n}}^{x}[|e_{y}
^{n+1+\delta_{n+1}^{+}}|+|e_{z}^{n+1+\delta_{n+1}^{+}}|]+L\Delta t|e_{z}
^{n}|\\
\text{ \ \ \ \ \ \ \ }+L\Delta t\mathbb{E}_{t_{n}}^{x}[|e_{y}^{n+\delta
_{n}^{-}}|+|e_{z}^{n+\delta_{n}^{-}}|+|e_{y}^{n+\delta_{n}^{+}}|+|e_{z}
^{n+\delta_{n}^{+}}|]+|R_{y}^{n,\delta}|,
\end{array}
\label{4.15}
\end{equation}
where $L$ is the Lipschitz constant of $f\left(  t,y_{1},z_{1},y_{2}
,z_{2}\right)  $ with respect to $y_{1}$, $z_{1}$, $y_{2}$ and $z_{2}$. Taking
square on both side of $\left(  \ref{4.15}\right)  $, using the inequalities
$(a+b)^{2}\leq a^{2}+b^{2}+(1+\gamma\Delta t)a^{2}+(1+\frac{1}{\gamma\Delta
t})b^{2}$ and $(\sum_{i=1}^{m}a_{i})^{2}\leq m\sum_{i=1}^{m}a_{i}^{2}$, we get
\begin{equation}
\begin{array}
[c]{l}
|e_{y}^{n}|^{2}\leq(1+\gamma\Delta t)\left\vert \mathbb{E}_{t_{n}}^{x}
[e_{y}^{n+1}]\right\vert ^{2}+(1+\frac{1}{\gamma\Delta t})\{4L^{2}\Delta
t^{2}|e_{z}^{n}|^{2}\\
\text{ \ \ \ \ \ \ \ \ }+24L^{2}\left(  1+L\right)  ^{2}\Delta t^{2}
\mathbb{E}_{t_{n}}^{x}[|e_{y}^{n+1}|^{2}+|e_{z}^{n+1}|^{2}]\\
\text{ \ \ \ \ \ \ \ \ }+24L^{2}\left(  1+L\right)  ^{2}\Delta t^{2}
\mathbb{E}_{t_{n}}^{x}[|e_{y}^{n+1+\delta_{n+1}^{-}}|^{2}+|e_{z}
^{n+1+\delta_{n+1}^{-}}|^{2}]\\
\text{ \ \ \ \ \ \ \ \ }+24L^{2}\left(  1+L\right)  ^{2}\Delta t^{2}
\mathbb{E}_{t_{n}}^{x}[|e_{y}^{n+1+\delta_{n+1}^{+}}|^{2}+|e_{z}
^{n+1+\delta_{n+1}^{+}}|^{2}]\\
\text{ \ \ \ \ \ \ \ \ }+16L^{2}\Delta t^{2}\mathbb{E}_{t_{n}}^{x}
[|e_{y}^{n+\delta_{n}^{-}}|^{2}+|e_{z}^{n+\delta_{n}^{-}}|^{2}]\\
\text{ \ \ \ \ \ \ \ \ }+16L^{2}\Delta t^{2}\mathbb{E}_{t_{n}}^{x}
[|e_{y}^{n+\delta_{n}^{+}}|^{2}+|e_{z}^{n+\delta_{n}^{+}}|^{2}]+4|R_{y}
^{n,\delta}|^{2}\}.
\end{array}
\label{4.16}
\end{equation}

Step2:The estimate of $e_{z}^{n}$.

In the same way above, subtracting $\left(  \ref{3.14ref}\right)  $ from
$\left(  \ref{scheme3.2.1}\right)  $, we have
\begin{equation}
\begin{array}
[c]{l}
\theta_{3}\Delta te_{z}^{n}=\theta_{3}\mathbb{E}_{t_{n}}^{x}[e_{y}^{n+1}\Delta
W_{t_{n+1}}]+\left(  1-\theta_{2}\right)  k_{n+1}\Delta t\mathbb{E}_{t_{n}
}^{x}[e_{f_{-}}^{n+1}\Delta W_{t_{n+1}}]\\
\text{ \ \ \ \ \ \ \ \ \ \ \ }+\left(  1-\theta_{2}\right)  \left(
1-k_{n+1}\right)  \Delta t\mathbb{E}_{t_{n}}^{x}[e_{f_{+}}^{n+1}\Delta
W_{t_{n+1}}]+R_{z}^{n,\delta}.
\end{array}
\label{4.17}
\end{equation}
Moreover,
\begin{equation}
\begin{array}
[c]{ll}
\theta_{3}\Delta t|e_{z}^{n}| & \leq\theta_{3}\left\vert \mathbb{E}_{t_{n}
}^{x}[\left(  e_{y}^{n+1}-\mathbb{E}_{t_{n}}^{x}[e_{y}^{n+1}]\right)  \Delta
W_{t_{n+1}}]\right\vert +|R_{z}^{n,\delta}|\\
& \text{ \ }+(1-\theta_{2})L\Delta t\mathbb{E}_{t_{n}}^{x}[(|e_{y}
^{n+1}|+|e_{z}^{n+1}|)|\Delta W_{t_{n+1}}|]\\
& \text{ \ }+(1-\theta_{2})L\Delta t\mathbb{E}_{t_{n}}^{x}[(|e_{y}
^{n+1+\delta_{n+1}^{-}}|+|e_{z}^{n+1+\delta_{n+1}^{-}}|)|\Delta W_{t_{n+1}
}|]\\
& \text{ \ }+(1-\theta_{2})L\Delta t\mathbb{E}_{t_{n}}^{x}[(|e_{y}
^{n+1+\delta_{n+1}^{+}}|+|e_{z}^{n+1+\delta_{n+1}^{+}}|)|\Delta W_{t_{n+1}}|].
\end{array}
\label{4.18}
\end{equation}
Applying H\"{o}lder inequality, we deduce that
\begin{equation}
\text{ \ }(\mathbb{E}_{t_{n}}^{x}[\left(  e_{y}^{n+1}-\mathbb{E}_{t_{n}}
^{x}[e_{y}^{n+1}]\right)  \Delta W_{t_{n+1}}])^{2}\leq\Delta t(\mathbb{E}
_{t_{n}}^{x}[|e_{y}^{n+1}|^{2}]-|\mathbb{E}_{t_{n}}^{x}[e_{y}^{n+1}]|^{2}),
\end{equation}
and
\begin{equation}
\begin{array}
[c]{l}
\text{ \ }(\mathbb{E}_{t_{n}}^{x}[(|e_{y}^{n+1}|+|e_{z}^{n+1}|)|\Delta
W_{t_{n+1}}|])^{2}\leq2\Delta t\mathbb{E}_{t_{n}}^{x}[|e_{y}^{n+1}|^{2}
+|e_{z}^{n+1}|^{2}],\\
\text{ \ }(\mathbb{E}_{t_{n}}^{x}[(|e_{y}^{n+1+\delta_{n+1}^{-}}
|+|e_{z}^{n+1+\delta_{n+1}^{-}}|+|e_{y}^{n+1+\delta_{n+1}^{+}}|+|e_{z}
^{n+1+\delta_{n+1}^{+}}|)|\Delta W_{t_{n+1}}|])^{2}\\
\;\ \ \ \ \ \ \ \ \ \leq4\Delta t\mathbb{E}_{t_{n}}^{x}[|e_{y}^{n+1+\delta
_{n+1}^{-}}|^{2}+|e_{z}^{n+1+\delta_{n+1}^{-}}|^{2}+|e_{y}^{n+1+\delta
_{n+1}^{+}}|^{2}+|e_{z}^{n+1+\delta_{n+1}^{+}}|^{2}].
\end{array}
\end{equation}
Then we have the estimate
\begin{equation}
\begin{array}
[c]{l}
\theta_{3}^{2}\Delta t^{2}|e_{z}^{n}|^{2}\leq2\theta_{3}^{2}\Delta t\left(
\mathbb{E}_{t_{n}}^{x}[|e_{y}^{n+1}|^{2}]-(\mathbb{E}_{t_{n}}^{x}[e_{y}
^{n+1}])^{2}\right) \\
\text{ \ \ \ \ \ \ \ \ \ \ \ \ \ \ \ }+12L^{2}\Delta t^{3}\mathbb{E}_{t_{n}
}^{x}[|e_{y}^{n+1}|^{2}+|e_{z}^{n+1}|^{2}]\\
\text{ \ \ \ \ \ \ \ \ \ \ \ \ \ \ \ }+24L^{2}\Delta t^{3}\mathbb{E}_{t_{n}
}^{x}[|e_{y}^{n+1+\delta_{n+1}^{-}}|^{2}+|e_{z}^{n+1+\delta_{n+1}^{-}}|^{2}]\\
\text{ \ \ \ \ \ \ \ \ \ \ \ \ \ \ \ }+24L^{2}\Delta t^{3}\mathbb{E}_{t_{n}
}^{x}[|e_{y}^{n+1+\delta_{n+1}^{+}}|^{2}+|e_{z}^{n+1+\delta_{n+1}^{+}}
|^{2}]+6|R_{z}^{n,\delta}|^{2}.
\end{array}
\label{4.20}
\end{equation}
Dividing both side of $\left(  \ref{4.20}\right)  $ by $2\theta_{3}^{2}\Delta
t$ we have
\begin{equation}
\begin{array}
[c]{l}
\frac{1}{2}\Delta t|e_{z}^{n}|^{2}\leq\mathbb{E}_{t_{n}}^{x}[|e_{y}^{n+1}
|^{2}]-(\mathbb{E}_{t_{n}}^{x}[e_{y}^{n+1}])^{2}\\
\text{\ \ \ \ \ \ \ \ \ \ \ \ \ \ }+\frac{6L^{2}}{\theta_{3}^{2}}\Delta
t^{2}\mathbb{E}_{t_{n}}^{x}[|e_{y}^{n+1}|^{2}+|e_{z}^{n+1}|^{2}]\\
\text{\ \ \ \ \ \ \ \ \ \ \ \ \ \ }+\frac{12L^{2}}{\theta_{3}^{2}}\Delta
t^{2}\mathbb{E}_{t_{n}}^{x}[|e_{y}^{n+1+\delta_{n+1}^{-}}|^{2}+|e_{z}
^{n+1+\delta_{n+1}^{-}}|^{2}]\\
\text{\ \ \ \ \ \ \ \ \ \ \ \ \ \ }+\frac{12L^{2}}{\theta_{3}^{2}}\Delta
t^{2}\mathbb{E}_{t_{n}}^{x}[|e_{y}^{n+1+\delta_{n+1}^{+}}|^{2}+|e_{z}
^{n+1+\delta_{n+1}^{+}}|^{2}]+\text{ }\frac{3}{\theta_{3}^{2}\Delta t}
|R_{z}^{n,\delta}|^{2}.
\end{array}
\label{4.21}
\end{equation}

Step3: The proof of $\left(  \ref{theorem4.1}\right)  $.

United $\left(\ref{4.16}\right)  $ and $\left(  \ref{4.21}\right)  $, it is direct to
obtain
\begin{equation}
\begin{array}
[c]{l}
|e_{y}^{n}|^{2}+\frac{1}{2}\Delta t|e_{z}^{n}|^{2}\\
\text{ \ \ \ \ \ }\leq(1+\gamma\Delta t)|\mathbb{E}_{t_{n}}^{x}[e_{y}
^{n+1}]|^{2}+\mathbb{E}_{t_{n}}^{x}[|e_{y}^{n+1}|^{2}]-(\mathbb{E}_{t_{n}}
^{x}[e_{y}^{n+1}])^{2}\\
\text{ \ \ \ \ \ \ \ \ \ }+(1+\frac{1}{\gamma\Delta t})24L^{2}(1+L)^{2}\Delta
t^{2}\mathbb{E}_{t_{n}}^{x}[|e_{y}^{n+1}|^{2}+|e_{z}^{n+1}|^{2}]\\
\text{ \ \ \ \ \ \ \ \ \ }+(1+\frac{1}{\gamma\Delta t})24L^{2}(1+L)^{2}\Delta
t^{2}\mathbb{E}_{t_{n}}^{x}[|e_{y}^{n+1+\delta_{n+1}^{-}}|^{2}+|e_{z}
^{n+1+\delta_{n+1}^{-}}|^{2}]\\
\text{ \ \ \ \ \ \ \ \ \ }+(1+\frac{1}{\gamma\Delta t})24L^{2}(1+L)^{2}\Delta
t^{2}\mathbb{E}_{t_{n}}^{x}[|e_{y}^{n+1+\delta_{n+1}^{+}}|^{2}+|e_{z}
^{n+1+\delta_{n+1}^{+}}|^{2}]\\
\text{ \ \ \ \ \ \ \ \ \ }+(1+\frac{1}{\gamma\Delta t})16L^{2}\Delta
t^{2}\mathbb{E}_{t_{n}}^{x}[|e_{y}^{n+\delta_{n}^{-}}|^{2}+|e_{z}
^{n+\delta_{n}^{-}}|^{2}]\\
\text{ \ \ \ \ \ \ \ \ \ }+(1+\frac{1}{\gamma\Delta t})16L^{2}\Delta
t^{2}\mathbb{E}_{t_{n}}^{x}[|e_{y}^{n+\delta_{n}^{+}}|^{2}+|e_{z}
^{n+\delta_{n}^{+}}|^{2}]\\
\text{ \ \ \ \ \ \ \ \ \ }+\frac{6L^{2}}{\theta_{3}^{2}}\Delta t^{2}
\mathbb{E}_{t_{n}}^{x}[|e_{y}^{n+1}|^{2}+|e_{z}^{n+1}|^{2}]+(1+\frac{1}
{\gamma\Delta t})4L^{2}\Delta t^{2}|e_{z}^{n}|^{2}\\
\text{ \ \ \ \ \ \ \ \ \ }+\frac{12L^{2}}{\theta_{3}^{2}}\Delta t^{2}
\mathbb{E}_{t_{n}}^{x}[|e_{y}^{n+1+\delta_{n+1}^{-}}|^{2}+|e_{z}
^{n+1+\delta_{n+1}^{-}}|^{2}]\\
\text{ \ \ \ \ \ \ \ \ \ }+\frac{12L^{2}}{\theta_{3}^{2}}\Delta t^{2}
\mathbb{E}_{t_{n}}^{x}[|e_{y}^{n+1+\delta_{n+1}^{+}}|^{2}+|e_{z}
^{n+1+\delta_{n+1}^{+}}|^{2}]\\
\text{ \ \ \ \ \ \ \ \ \ }+(1+\frac{1}{\gamma\Delta t})4|R_{y}^{n,\delta}
|^{2}+\frac{3}{\theta_{3}^{2}\Delta t}|R_{z}^{n,\delta}|^{2},
\end{array}
\label{4.22}
\end{equation}
which implies
\begin{equation}
\begin{array}
[c]{l}
|e_{y}^{n}|^{2}+\Delta t|e_{z}^{n}
|^{2}\left(  \frac{1}{2}-\frac{4L^{2}}{\gamma}-4L^{2}\Delta
t\right)  \ \ \ \ \ \ \ \ \ \ \ \ \ \ \\
\;\ \ \ \ \ \ \ \ \ \ \leq(1+\gamma\Delta t)\mathbb{E}_{t_{n}}^{x}
[|e_{y}^{n+1}|^{2}]\\
\text{ \ \ \ \ \ \ \ \ \ \ \ \ \ }+(\frac{D_{2}}{\gamma}+D_{1}\Delta t)\Delta
t\mathbb{E}_{t_{n}}^{x}[|e_{y}^{n+1}|^{2}+|e_{z}^{n+1}|^{2}]\\
\text{ \ \ \ \ \ \ \ \ \ \ \ \ \ }+(\frac{D_{2}}{\gamma}+D_{1}\Delta t)\Delta
t\mathbb{E}_{t_{n}}^{x}[|e_{y}^{n+\delta_{n}^{-}}|^{2}+|e_{z}^{n+\delta
_{n}^{-}}|^{2}]\\
\text{ \ \ \ \ \ \ \ \ \ \ \ \ \ }+(\frac{D_{2}}{\gamma}+D_{1}\Delta t)\Delta
t\mathbb{E}_{t_{n}}^{x}[|e_{y}^{n+\delta_{n}^{+}}|^{2}+|e_{z}^{n+\delta
_{n}^{+}}|^{2}]\\
\text{ \ \ \ \ \ \ \ \ \ \ \ \ \ }+(\frac{D_{2}}{\gamma}+D_{1}\Delta t)\Delta
t\mathbb{E}_{t_{n}}^{x}[|e_{y}^{n+1+\delta_{n+1}^{-}}|^{2}+|e_{z}
^{n+1+\delta_{n+1}^{-}}|^{2}]\\
\text{ \ \ \ \ \ \ \ \ \ \ \ \ \ }+(\frac{D_{2}}{\gamma}+D_{1}\Delta t)\Delta
t\mathbb{E}_{t_{n}}^{x}[|e_{y}^{n+1+\delta_{n+1}^{+}}|^{2}+|e_{z}
^{n+1+\delta_{n+1}^{+}}|^{2}]\\
\text{ \ \ \ \ \ \ \ \ \ \ \ \ \ }+(1+\frac{1}{\gamma\Delta t})4|R_{y}^{n,\delta}|^{2}+
\frac{3}{\theta_{3}^{2}\Delta t}|R_{z}^{n,\delta}|^{2},
\end{array}
\label{4.23}
\end{equation}
where $D_{1}=24L^{2}(1+L)^{2}+\frac{12L^{2}}{\theta_{3}^{2}}$, $D_{2}
=24L^{2}(1+L)^{2}$. From $\left(  \ref{4.23}\right)  $ we have
\begin{equation}
\begin{array}
[c]{l}
|e_{y}^{n}|^{2}+C_{1}^{\ast}\Delta t|e_{z}^{n}|^{2}\leq(1+\gamma\Delta
t)\mathbb{E}_{t_{n}}^{x}[|e_{y}^{n+1}|^{2}]\\
\;\ \ \ \ \ \ \ \ \ \ \ \ \ \ \ \ \ \ \ \ \ \ \ \ \ +(\frac{D_{2}}{\gamma
}+D_{1}\Delta t)\Delta t\sum_{i\in\mathcal{I}_{n}^{\ast}}\mathbb{E}
_{t_{n}}^{x}\mathbb{[}|e_{y}^{n+i}|^{2}+|e_{z}^{n+i}|^{2}]\\
\;\ \ \ \ \ \ \ \ \ \ \ \ \ \ \ \ \ \ \ \ \ \ \ \ \ +(1+\frac{1}{\gamma\Delta t})4
|R_{y}^{n,\delta}|^{2}+\frac{3}{\theta_{3}^{2}\Delta t}|R_{z}^{n,\delta}|^{2},
\end{array}
\label{4.25}
\end{equation}
where $C_{1}^{\ast}=\frac{1}{2}-\frac{4L^{2}}{\gamma}-4L^{2}\Delta t$,
$\mathcal{I}_{n}^{\ast}=\left\{  1,\delta_{n}^{-},\delta_{n}^{+}
,\delta_{n+1}^{-}+1,\delta_{n+1}^{+}+1\right\}  $.  While $n\,<n_{0}$, $\delta
_{n}^{-}=0$, by Remark \ref{remark}, similar to the process above, we have
\begin{equation}
\begin{array}
[c]{l}
|e_{y}^{n}|^{2}+C_{1}^{\ast\ast}\Delta t|e_{z}^{n}|^{2}\leq(1+\gamma\Delta
t)\mathbb{E}_{t_{n}}^{x}[|e_{y}^{n+1}|^{2}]\ \\
\;\ \ \ \ \ \ \ \ \ \ \ \ \ \ \ \ \ \ \ \ \ \ \ \ \ +2(\frac{D_{2}}{\gamma
}+D_{1}\Delta t)\Delta t\sum_{i\in\mathcal{I}_{n}^{\ast\ast}}\mathbb{E}_{t_{n}}^{x}
[|e_{y}^{n+i}|^{2}+|e_{z}^{n+i}|^{2}]\\
\;\ \ \ \ \ \ \ \ \ \ \ \ \ \ \ \ \ \ \ \ \ \ \ \ \ +(1+\frac{1}{\gamma\Delta t})4
|R_{y}^{n,\delta}|^{2}+\frac{3}{\theta_{3}^{2}\Delta t}|R_{z}^{n,\delta}|^{2},
\end{array}
\label{4.24}
\end{equation}
where $C_{1}^{\ast\ast}=\frac{1}{2}-\frac{8L^{2}}{\gamma}-8L^{2}\Delta t$,
$\mathcal{I}_{n}^{\ast\ast}=\left\{  1,\delta_{n}^{+},\delta_{n+1}^{-}
+1,\delta_{n+1}^{+}+1\right\}  $. When $\Delta t$ is
sufficiently small, by letting $\gamma$ be big enough, we can choose constants
$C_{1},C_{2},C_{3},C_{4}$, which satisfy
\[
\begin{array}
[c]{ll}
C_{1}\leq\min\left\{  C_{1}^{\ast},C_{1}^{\ast\ast}\right\}  , &
\;\ 2(\frac{D_{2}}{\gamma}+D_{1}\Delta t)\leq C_{3},\\
C_{2}\geq\gamma, & \;\ C_{1}-3C_{3}\geq C_{4}>0.
\end{array}
\]
Taking mathematical expectation on both sides of $\left(  \ref{4.25}\right)  $ and
$\left(  \ref{4.24}\right)$, rewriting the inequality, we have
\begin{equation}
\begin{array}
[c]{l}
\mathbb{E[}|e_{y}^{n}|^{2}]+C_{1}\Delta t\mathbb{E[}|e_{z}^{n}|^{2}
]\leq\left(  1+C_{2}\Delta t\right)  \mathbb{E[}|e_{y}^{n+1}|^{2}]\\
\;\ \ \ \ \ \ \ \ \ \ \ \ \ \ \ \ \ \ \ \ \ \ \ \ \ \ \ \ \ \ \ \ +C_{3}\Delta
t\sum_{i\in\mathcal{I}_{n}}\mathbb{E[}|e_{y}^{n+i}|^{2}+|e_{z}^{n+i}
|^{2}]+R^{n},
\end{array}
\label{4.26}
\end{equation}
where $\mathcal{I}_{n}=\left\{  1,\delta_{n}^{-},\delta_{n}^{+},\delta
_{n+1}^{-}+1,\delta_{n+1}^{+}+1\right\}  \backslash\left\{  0\right\}$  and
$R^n=(1+\frac{1}{\gamma\Delta t})4$$\mathbb{E}[|R_{y}^{n,\delta}|^{2}]
+\frac{3}{\theta_{3}^{2}\Delta t}\mathbb{E}[|R_{z}^{n,\delta}|^{2}]$.
Applying Lemma \ref{lemma} to $\left(  \ref{4.26}\right)  $, the conclusion
$\left(  \ref{theorem4.1}\right)  \;$can be drawn.
\end{proof}
\end{theorem}

\begin{remark}
Theorem \ref{Theorem 4.1} implies that our explicit $\theta$-schemes are stable.
\end{remark}

\subsection{Error estimates for anticipated BSDEs \label{subsec4.2}}

In this section, we discuss the error estimates of the explicit $\theta
$-schemes for anticipate BSDEs. Under some regularity conditions on $f$,
$\varphi$ and $\delta(t)$, we first give the estimates of the local truncation
errors $R_{y}^{n,\delta}$ and $R_{z}^{n,\delta}$ defined in $\left(
\ref{3.8ref}\right)  $ and $\left(  \ref{3.14ref}\right)  $, then based on the
Theorem \ref{Theorem 4.1}, we get the error estimates of the explicit $\theta
$-schemes \ref{scheme}.

\begin{lemma}
\label{truncation error lemma}Let $R_{y}^{n,\delta}$ and $R_{z}^{n,\delta}$ be
the local truncation errors derived from $\left(  \ref{3.8ref}\right)  $ and
$\left(  \ref{3.14ref}\right)  $, respectively. Assume $\sqrt{\delta(t)}$ is
smooth enough, for sufficiently small $\Delta t$, it holds that:

\begin{enumerate}
\item For $\theta_{i}\left(  i=1,2\right)  \in\lbrack0,1]$, $\theta_{3}
\in(0,1]$, if $\varphi\in C_{b}^{\frac{1}{2},2}$, $f\in C_{b}^{\frac{1}
{2},1,1,1,1}$, we have
\begin{equation}
\left\vert R_{y}^{n,\delta}\right\vert \leq C(\Delta t)^{\frac{3}{2}
},\;\left\vert R_{z}^{n,\delta}\right\vert \leq C(\Delta t)^{\frac{3}{2}},
\label{truncation cond1}
\end{equation}

and if $\varphi\in C_{b}^{1,4}$, $f\in C_{b}^{1,3,3,3,3}$, we have
\begin{equation}
\left\vert R_{y}^{n,\delta}\right\vert \leq C(\Delta t)^{2},\;\left\vert
R_{z}^{n,\delta}\right\vert \leq C(\Delta t)^{2}, \label{truncation cond2}
\end{equation}

\item In particular, for $\theta_{i}=\frac{1}{2}\left(  i=1,2,3\right)  $, if
$\varphi\in C_{b}^{2,6}$, $f\in C_{b}^{2,5,5,5,5}$, we have
\begin{equation}
\left\vert R_{y}^{n,\delta}\right\vert \leq C(\Delta t)^{3},\;\left\vert
R_{z}^{n,\delta}\right\vert \leq C(\Delta t)^{3}, \label{truncation cond3}
\end{equation}

\end{enumerate}
where $C>0$ is a constant just depending on $T$, the upper bounds of
derivatives of $\varphi$, $\delta(t)$ and $f$.

\begin{proof}
1.According to the definition of $R_{y}^{n,\delta}$ and $R_{z}^{n,\delta},$
using the continuity of the $\varphi$ and $f,$ $\left(  \ref{truncation cond1}
\right)  $ can be proved easily. 2.Now we prove $\left(
\ref{truncation cond2}\right)  $ and $\left(  \ref{truncation cond3}\right)
$. We first approximate the truncation errors $R_{y_{1}}^{n,\delta}$,
$R_{z_{1}}^{n,\delta}$ and $R_{z_{3}}^{n,\delta}$. By $\left(  \ref{1.3}
\right)  $, we can define $f(t,Y_{t},Z_{t},Y_{t+\delta(t)},Z_{t+\delta
(t)})=f(t,u(t,W_{t}),\nabla_{x}u(t,W_{t}),u(t+\delta(t),W_{t+\delta
(t)}),\nabla_{x}u(t+\delta(t),W_{t+\delta(t)}))=F(t,W_{t},W_{t+\delta(t)})$.
If $f\in C_{b}^{1+k,3+2k,3+2k,3+2k,3+2k}$, $\varphi\in C_{b}^{1+k,4+2k}$ and $\sqrt{\delta(t)}\in C_{b}^{1+k} $,
then $F\in C_{b}^{1+k,2+2k,2+2k}$, $k=0,1,\ldots$. We rewrite $R_{y_{1}}^{n,\delta}$ as
follow
\begin{equation}
\begin{array}
[c]{l}
R_{y_{1}}^{n,\delta}=\int_{t_{n}}^{t_{n+1}}\mathbb{E}_{t_{n}}^{x}\left[
F\left(  s,W_{s},W_{s+\delta(s)}\right)  \right]  ds-\theta_{1}\mathbb{E}
_{t_{n}}^{x}\left[  F\left(  t_{n},W_{t_{n}},W_{t_{n}+\delta(t_{n})}\right)
\right]  \Delta t\\
\text{ \ \ \ \ \ \ \ \ }-\left(  1-\theta_{1}\right)  \mathbb{E}_{t_{n}}
^{x}\left[  F\left(  t_{n+1},W_{t_{n+1}},W_{t_{n+1}+\delta(t_{n+1})}\right)
\right]  \Delta t.
\end{array}
\label{4.29}
\end{equation}
By the properties of conditional mathematical expectation we can see
\begin{equation}
\begin{array}
[c]{l}
\mathbb{E}_{t_{n}}^{x}\left[  F\left(  s,W_{s},W_{s+\delta(s)}\right)
\right]  =\mathbb{E}_{t_{n}}^{x}\left[  \mathbb{E}\left[  \left.  F\left(
s,W_{s},W_{s}+W_{s+\delta(s)}-W_{s}\right)  \right\vert \mathcal{F}
_{s}\right]  \right] \\
\text{ \ \ \ \ \ \ \ \ \ \ \ \ \ \ \ \ \ \ \ \ \ \ \ \ \ \ \ \ }
=\mathbb{E}_{t_{n}}^{x}\left[  \left.  \mathbb{E}\left[  F\left(
t,x,y+W_{s+\delta(s)}-W_{s}\right)  \right]  \right\vert _{t=s,x=y=W_{s}
}\right]  .
\end{array}
\label{4.30}
\end{equation}
Similarly, we obtain
\begin{equation}
\begin{array}
[c]{l}
\mathbb{E}_{t_{n}}^{x}\left[  F\left(  t_{n},W_{t_{n}},W_{t_{n}+\delta(t_{n}
)}\right)  \right] \\
\text{\ \ \ \ \ }=\mathbb{E}_{t_{n}}^{x}\left[  \left.  \mathbb{E}\left[
F\left(  t,x,y+W_{t_{n}+\delta(t_{n})}-W_{t_{n}}\right)  \right]  \right\vert
_{t=t_{n},x=y=W_{t_{n}}}\right]  ,
\end{array}
\label{4.31}
\end{equation}
and
\begin{equation}
\begin{array}
[c]{l}
\mathbb{E}_{t_{n}}^{x}\left[  F\left(  t_{n+1},W_{t_{n+1}},W_{t_{n+1}
+\delta(t_{n+1})}\right)  \right] \\
\text{ \ \ \ }=\mathbb{E}_{t_{n}}^{x}\left[  \left.  \mathbb{E}\left[
F\left(  t,x,y+W_{t_{n+1}+\delta(t_{n+1})}-W_{t_{n+1}}\right)  \right]
\right\vert _{t=t_{n+1},x=y=W_{t_{n+1}}}\right]  .
\end{array}
\label{4.32}
\end{equation}
Now we define a new function
\begin{equation}
\begin{array}
[c]{l}
H\left(  t,x,y\right)  =\mathbb{E}\left[  F\left(  t,x,y+W_{\delta(t)}\right)
\right] \\
\text{ \ \ \ \ \ \ \ \ \ \ \ }=\frac{1}{\sqrt{2\pi}}\int_{-\infty}^{+\infty
}F(t,x,y+\sqrt{\delta(t)}z)e^{-\frac{\;\;z^{2}}{2}}dz.
\end{array}
\label{4.33}
\end{equation}
Since $F\in C_{b}^{1+k,2+2k,2+2k}$ and $\int_{-\infty}^{+\infty}
e^{-\frac{\;\;z^{2}}{2}}dz<\infty$, we can draw the conclusion that $H\in$
$C_{b}^{1+k,2+2k,2+2k}$. Due to the properties of Brownian motion, It's worth
noting that
\[
\mathbb{E}\left[  F\left(  t,x,y+W_{\delta(t)}\right)  \right]  =\mathbb{E}
\left[  F\left(  t,x,y+W_{t+\delta(t)}-W_{t}\right)  \right]  ,
\]
then
\begin{equation}
\begin{array}
[c]{l}
\mathbb{E}_{t_{n}}^{x}\left[  F\left(  s,W_{s},W_{s+\delta(s)}\right)
\right]  =\mathbb{E}_{t_{n}}^{x}\left[  H\left(  s,W_{s},W_{s}\right)
\right]  ,\\
\mathbb{E}_{t_{n}}^{x}\left[  F\left(  t_{n},W_{t_{n}},W_{t_{n}+\delta(t_{n}
)}\right)  \right]  =\mathbb{E}_{t_{n}}^{x}\left[  H\left(  t_{n},W_{t_{n}
},W_{t_{n}}\right)  \right]  ,\\
\mathbb{E}_{t_{n}}^{x}\left[  F\left(  t_{n+1},W_{t_{n+1}},W_{t_{n+1}
+\delta(t_{n+1})}\right)  \right]  =\mathbb{E}_{t_{n}}^{x}\left[  H\left(
t_{n+1},W_{t_{n}+1},W_{t_{n+1}}\right)  \right]  .
\end{array}
\label{4.34}
\end{equation}
Therefore, $\left(  \ref{4.29}\right)  $ changes to be
\begin{align}
R_{y_{1}}^{n,\delta}  &  =\int_{t_{n}}^{t_{n+1}}\mathbb{E}_{t_{n}}^{x}\left[
H\left(  s,W_{s},W_{s}\right)  \right]  ds-\theta_{1}\Delta t\mathbb{E}
_{t_{n}}^{x}\left[  H\left(  t_{n},W_{t_{n}},W_{t_{n}}\right)  \right]
\label{4.35}\\
&  \;\ \ -\left(  1-\theta_{1}\right)  \Delta t\mathbb{E}_{t_{n}}^{x}\left[
H\left(  t_{n+1},W_{t_{n+1}},W_{t_{n+1}}\right)  \right]  .\nonumber
\end{align}
We denote $H\left(  s,W_{s},W_{s}\right)  $ by $H_{s}$, by It$\hat{o}$ formula, we know
\begin{equation}
H_{s}=H_{t_{n}}+\int_{t_{n}}^{s}L^{0}H_{r}dr+\int_{t_{n}}^{s}L^{1}H_{r}dW_{r},
\end{equation}
where\ $L^{0}=\frac{\partial}{\partial t}+\frac{1}{2}(\frac{\partial}{\partial
x}+\frac{\partial}{\partial y})^{2}$, $L^{1}=\frac{\partial}{\partial x}
+\frac{\partial}{\partial y}$. If $\varphi\in C_{b}^{1,4}$ and $f\in
C_{b}^{1,3,3,3,3}$, then $H\in C_{b}^{1,2,2}$, we can deduce
\begin{equation}
\begin{array}
[c]{l}
\left\vert R_{y_{1}}^{n,\delta}\right\vert =\left\vert \int_{t_{n}}^{t_{n+1}
}\mathbb{E}_{t_{n}}^{x}[H_{t_{n}}+\int_{t_{n}}^{s}L^{0}H_{r}dr+\int_{t_{n}
}^{s}L^{1}H_{r}dW_{r}]ds-\theta_{1}\Delta t\mathbb{E}_{t_{n}}^{x}\left[
H_{t_{n}}\right]  \right. \\
\text{ \ \ \ \ \ \ \ \ \ \ }-\left.  \left(  1-\theta_{1}\right)  \Delta
t\mathbb{E}_{t_{n}}^{x}[H_{t_{n}}+\int_{t_{n}}^{t_{n+1}}L^{0}H_{r}
dr+\int_{t_{n}}^{t_{n+1}}L^{1}H_{r}dW_{r}]\right\vert \\
\text{ \ \ \ \ \ \ }=\left\vert \int_{t_{n}}^{t_{n+1}}\left(  \int_{t_{n}}
^{s}\mathbb{E}_{t_{n}}^{x}\left[  L^{0}H_{r}\right]  dr-\left(  1-\theta
_{1}\right)  \int_{t_{n}}^{t_{n+1}}\mathbb{E}_{t_{n}}^{x}\left[  L^{0}
H_{r}\right]  dr\right)  ds\right\vert \\
\text{ \ \ \ \ \ \ }\leq C\left(  \Delta t\right)  ^{2}.
\end{array}
\label{4.36}
\end{equation}
Furthermore, if $\theta_{1}=\frac{1}{2}$, $\varphi\in C_{b}^{2,6}$ and $f\in
C_{b}^{2,5,5,5,5}$, then $H\in C_{b}^{2,4,4}$, continue to apply It$\hat{o}
$ formula to the $L^{0}H_{r}$ in equation $\left(  \ref{4.36}\right)
$, we can deduce
\[
\begin{array}
[c]{l}
\left\vert R_{y_{1}}^{n,\delta}\right\vert =\left\vert \int_{t_{n}}^{t_{n+1}
}\left(  \int_{t_{n}}^{s}\mathbb{E}_{t_{n}}^{x}\left[  L^{0}H_{t_{n}}
+\int_{t_{n}}^{r}L^{0}L^{0}H_{u}du+\int_{t_{n}}^{r}L^{1}L^{0}H_{u}
dW_{u}\right]  dr\right.  \right. \\
\text{ \ \ \ \ \ \ \ \ \ \ }-\left.  \left.  \frac{1}{2}\int_{t_{n}}^{t_{n+1}
}\mathbb{E}_{t_{n}}^{x}\left[  L^{0}H_{t_{n}}+\int_{t_{n}}^{r}L^{0}L^{0}
H_{u}du+\int_{t_{n}}^{r}L^{1}L^{0}H_{u}dW_{u}\right]  dr\right)  ds\right\vert
\\
\text{ \ \ \ \ \ \ }=\left\vert \mathbb{E}_{t_{n}}^{x}[L^{0}H_{t_{n}}]\left(
\int_{t_{n}}^{t_{n+1}}\int_{t_{n}}^{s}drds-\frac{1}{2}\int_{t_{n}}^{t_{n+1}
}\int_{t_{n}}^{t_{n+1}}drds\right)  \right. \\
\text{ \ \ \ \ \ \ \ \ \ \ }+\left.  \int_{t_{n}}^{t_{n+1}}\mathbb{E}_{t_{n}
}^{x}\left[  \int_{t_{n}}^{s}\int_{t_{n}}^{r}L^{0}L^{0}H_{u}dudr-\frac{1}
{2}\int_{t_{n}}^{t_{n+1}}\int_{t_{n}}^{r}L^{0}L^{0}H_{u}dudr\right]
ds\right\vert \\
\text{ \ \ \ \ \ \ }=\left\vert \int_{t_{n}}^{t_{n+1}}\mathbb{E}_{t_{n}}
^{x}\left[  \int_{t_{n}}^{s}\int_{t_{n}}^{r}L^{0}L^{0}H_{u}dudr-\frac{1}
{2}\int_{t_{n}}^{t_{n+1}}\int_{t_{n}}^{r}L^{0}L^{0}H_{u}dudr\right]
ds\right\vert \\
\text{ \ \ \ \ \ \ }\leq C\left(  \Delta t\right)  ^{3},
\end{array}
\]
then $\left\vert R_{y}^{n,\delta}\right\vert \leq C(\Delta t)^{3}$. In
addition, by It$\hat{o}^{\prime}$s isometry formula, we have
\begin{equation}
\begin{array}
[c]{l}
\left\vert R_{z_{1}}^{n,\delta}\right\vert =\left\vert \int_{t_{n}}^{t_{n+1}
}\mathbb{E}_{t_{n}}^{x}[H_{s}\Delta W_{t_{n+1}}]ds-(1-\theta_{2})\Delta
t\mathbb{E}_{t_{n}}^{x}[H_{t_{n+1}}\Delta W_{t_{n+1}}]\right\vert \\
\text{ \ \ \ \ \ \ }=\left\vert \int_{t_{n}}^{t_{n+1}}\mathbb{E}_{t_{n}}
^{x}[(H_{t_{n}}+\int_{t_{n}}^{s}L^{0}H_{r}dr+\int_{t_{n}}^{s}L^{1}H_{r}
dW_{r})\Delta W_{t_{n+1}}]ds\right. \\
\text{ \ \ \ \ \ \ \ \ \ \ }-\left.  \left(  1-\theta_{2}\right)  \Delta
t\mathbb{E}_{t_{n}}^{x}[(H_{t_{n}}+\int_{t_{n}}^{t_{n+1}}L^{0}H_{r}
dr+\int_{t_{n}}^{t_{n+1}}L^{1}H_{r}dW_{r})\Delta W_{t_{n+1}}]\right\vert \\
\text{ \ \ \ \ \ \ }=\left\vert \int_{t_{n}}^{t_{n+1}}\mathbb{E}_{t_{n}}
^{x}[\int_{t_{n}}^{s}L^{0}H_{r}dr\Delta W_{t_{n+1}}-\int_{t_{n}}^{s}L^{1}
H_{r}dr]ds\right. \\
\text{ \ \ \ \ \ \ \ \ \ \ }+(1-\theta_{2})\left.  \int_{t_{n}}^{t_{n+1}
}\mathbb{E}[\int_{t_{n}}^{t_{n+1}}L^{0}H_{r}dr\Delta W_{t_{n+1}}-\int_{t_{n}
}^{t_{n+1}}L^{1}H_{r}dr]ds\right\vert \\
\text{ \ \ \ \ \ \ }\leq R_{1}+R_{2},
\end{array}
\label{4.37}
\end{equation}
where
\[
\begin{array}
[c]{l}
R_{1}=\left\vert \int_{t_{n}}^{t_{n+1}}\mathbb{E}_{t_{n}}^{x}\left[  \left(
\int_{t_{n}}^{s}L^{0}H_{r}dr-(1-\theta_{2})\int_{t_{n}}^{t_{n+1}}L^{0}
H_{r}dr\right)  \Delta W_{t_{n+1}}\right]  ds\right\vert ,\\
R_{2}=\left\vert \int_{t_{n}}^{t_{n+1}}\left(  \int_{t_{n}}^{s}\mathbb{E}
_{t_{n}}^{x}\left[  L^{1}H_{r}\right]  dr-(1-\theta_{2})\int_{t_{n}}^{t_{n+1}
}\mathbb{E}_{t_{n}}^{x}\left[  L^{1}H_{r}\right]  dr\right)  ds\right\vert .
\end{array}
\]
Owing to $\mathbb{E}_{t_{n}}^{x}[|\Delta W_{t_{n+1}}|]\leq(\Delta t)^{\frac
{1}{2}},$ while $H\in C_{b}^{1,2,2}$, we have $\left\vert R_{1}\right\vert
\leq C(\Delta t)^{\frac{5}{2}}$ and $\left\vert R_{2}\right\vert \leq C(\Delta
t)^{2},$ then $\left\vert R_{z_{1}}^{n,\delta}\right\vert \leq C(\Delta
t)^{2}$. What's more, while $H\in C_{b}^{2,4,4}$, $\theta_{2}=\frac{1}{2}$,
using It$\hat{o}$ formula again
\begin{equation}
\begin{array}
[c]{l}
R_{1}=\left\vert \int_{t_{n}}^{t_{n+1}}\mathbb{E}_{t_{n}}^{x}[(\int_{t_{n}
}^{s}\int_{t_{n}}^{r}L^{0}L^{0}H_{u}dudr-\frac{1}{2}\int_{t_{n}}^{t_{n+1}}
\int_{t_{n}}^{r}L^{0}L^{0}H_{u}dudr)\Delta W_{t_{n+1}}]ds\right\vert \\
\text{ \ \ }\leq\sqrt{\Delta t}\left\vert \int_{t_{n}}^{t_{n+1}}
\mathbb{E}_{t_{n}}^{x}[(\int_{t_{n}}^{s}\int_{t_{n}}^{r}L^{0}L^{0}
H_{u}dudr-\frac{1}{2}\int_{t_{n}}^{t_{n+1}}\int_{t_{n}}^{r}L^{0}L^{0}
H_{u}dudr)^{2}]^{\frac{1}{2}}ds\right\vert \\
\text{ \ \ }\leq\Delta t\left\vert \int_{t_{n}}^{t_{n+1}}\mathbb{E}_{t_{n}
}^{x}[(\int_{t_{n}}^{s}\int_{t_{n}}^{r}L^{0}L^{0}H_{u}dudr-\frac{1}{2}
\int_{t_{n}}^{t_{n+1}}\int_{t_{n}}^{r}L^{0}L^{0}H_{u}dudr)^{2}]ds\right\vert
^{\frac{1}{2}}\\
\text{ \ \ }\leq C\left(  \Delta t\right)  ^{\frac{7}{2}},
\end{array}
\end{equation}
and
\begin{equation}
\begin{array}
[c]{l}
R_{2}=\left\vert \int_{t_{n}}^{t_{n+1}}\left(  \int_{t_{n}}^{s}\mathbb{E}
_{t_{n}}^{x}[L^{1}H_{t_{n}}+\int_{t_{n}}^{r}L^{0}L^{1}H_{u}du+\int_{t_{n}}
^{r}L^{1}L^{1}H_{u}dW_{u}]dr\right.  \right. \\
\text{ \ \ \ \ \ \ }-\left.  \left.  \frac{1}{2}\int_{t_{n}}^{t_{n+1}
}\mathbb{E}_{t_{n}}^{x}[L^{1}H_{t_{n}}+\int_{t_{n}}^{r}L^{0}L^{1}H_{u}
du+\int_{t_{n}}^{r}L^{1}L^{1}H_{u}dW_{u}]dr\right)  ds\right\vert \\
\text{ \ \ }=\left\vert \int_{t_{n}}^{t_{n+1}}\mathbb{E}_{t_{n}}^{x}\left[
\int_{t_{n}}^{s}\int_{t_{n}}^{r}L^{0}L^{1}H_{u}dudr-\frac{1}{2}\int_{t_{n}
}^{t_{n+1}}\int_{t_{n}}^{r}L^{0}L^{1}H_{u}dudr\right]  ds\right\vert \\
\text{ \ \ }\leq C(\Delta t)^{3},
\end{array}
\end{equation}
thus $|R_{z_{1}}^{n,\delta}|\leq C(\Delta t)^{3}$. Then it is straightforward
to show by a computation similar to $R_{y_{1}}^{n,\delta}$ and $R_{z_{1}
}^{n,\delta}$
\begin{equation}
R_{z_{3}}^{n,\delta}=-\int_{t_{n}}^{t_{n+1}}\int_{t_{n}}^{s}\mathbb{E}_{t_{n}
}^{x}\left[  L^{\ast}Z_{r}\right]  drds-\left(  1-\theta_{3}\right)
\int_{t_{n}}^{t_{n+1}}\int_{t_{n}}^{t_{n+1}}\mathbb{E}_{t_{n}}^{x}\left[
L^{\ast}Z_{r}\right]  drds,
\end{equation}
where $L^{\ast}=\frac{\partial}{\partial t}+\frac{1}{2}\frac{\partial^{2}
}{\partial x^{2}}$. We also have the estimates
\[
\begin{array}
[c]{ll}
|R_{z_{3}}^{n,\delta}|\leq C(\Delta t)^{2},\text{ \ } & \text{if }\varphi\in
C_{b}^{1,4},\\
|R_{z_{3}}^{n,\delta}|\leq C(\Delta t)^{3}, & \text{if }\varphi\in C_{b}
^{2,6}\text{ and}\ \theta_{3}=\frac{1}{2}.
\end{array}
\]
Next, we estimate the non-grid point interpolation errors $R_{y_{2}}
^{n,\delta}$ and $R_{z_{2}}^{n,\delta}$. Similar to $R_{y_{1}}^{n,\delta},$ we
rewrite $R_{y_{2}}^{n,\delta}$ and $R_{z_{2}}^{n,\delta}$ as follow
\begin{equation}
\begin{array}
[c]{l}
R_{y_{2}}^{n,\delta}=\theta_{1}\Delta t\mathbb{E}_{t_{n}}^{x}[F_{t_{n}
+\delta(t_{n})}-k_{n}F_{t_{n+\delta_{n}^{-}}}-\left(  1-k_{n}\right)
F_{t_{n+\delta_{n}^{+}}}]\\
\text{ \ \ \ \ \ \ \ }+(1-\theta_{1})\Delta t\mathbb{E}_{t_{n}}^{x}[\hat
{F}_{t_{n+1}+\delta(t_{n+1})}-k_{n+1}\hat{F}_{t_{n+1+\delta_{n+1}^{-}}
}-\left(  1-k_{n+1}\right)  \hat{F}_{t_{n+1+\delta_{n+1}^{+}}}],
\end{array}
\end{equation}
and
\begin{equation}
\begin{array}
[c]{l}
R_{z_{2}}^{n,\delta}=(1-\theta_{2})\Delta t\mathbb{E}_{t_{n}}^{x}[\hat
{F}_{t_{n+1}+\delta(t_{n+1})}\Delta W_{t_{n+1}}-k_{n+1}\hat{F}_{t_{n+1+\delta
_{n+1}^{-}}}\Delta W_{t_{n+1}}],\\
\text{ \ \ \ \ \ \ \ \ }-(1-\theta_{2})\left(  1-k_{n+1}\right)  \Delta
t\mathbb{E}_{t_{n}}^{x}[\hat{F}_{t_{n+1+\delta_{n+1}^{+}}}\Delta W_{t_{n+1}}]
\end{array}
\end{equation}
where we denote $F_{r}=F(t_{n},W_{t_{n}},W_{r})$ and $\hat{F}_{r}
=F(t_{n+1},W_{t_{n+1}},W_{r})$. if $F\in C_{b}^{1,2,2}$, we apply It$\hat{o}$
formula to the third variable of $F_{r}$ and $\hat{F}_{r}$, we obtain
\begin{equation}
\begin{array}
[c]{l}%
|R_{y_{2}}^{n,\delta}|\leq \Delta t\left \vert \mathbb{E}_{t_{n}}^{x}%
[F_{t_{n+\delta_{n}^{-}}}+\int_{t_{n+\delta_{n}^{-}}}^{t_{n}+\delta(t_{n}%
)}\frac{1}{2}F_{r}^{(2)}dr+\int_{t_{n+\delta_{n}^{-}}}^{t_{n}+\delta(t_{n}%
)}F_{r}^{(2)}dW_{r}]\right.  \\
\text{ \  \  \  \  \  \  \  \  \ }+\left.  \mathbb{E}_{t_{n}}^{x}[\hat{F}%
_{t_{n+1+\delta_{n+1}^{-}}}+\int_{t_{n+1+\delta_{n+1}^{-}}}^{t_{n+1}%
+\delta(t_{n+1})}\frac{1}{2}\hat{F}_{r}^{(2)}dr+\int_{t_{n+1+\delta_{n+1}^{-}%
}}^{t_{n+1}+\delta(t_{n+1})}\hat{F}_{r}^{(2)}dW_{r}]\right.  \\
\text{ \  \  \  \  \  \  \  \  \ }-\left.  \left(  1-k_{n}\right)  \mathbb{E}_{t_{n}%
}^{x}[F_{t_{n+\delta_{n}^{-}}}+\int_{t_{n+\delta_{n}^{-}}}^{t_{n+\delta
_{n}^{+}}}\frac{1}{2}F_{r}^{(2)}dr+\int_{t_{n+\delta_{n}^{-}}}^{t_{n+\delta
_{n}^{+}}}F_{r}^{(2)}dW_{r}]\right.  \\
\text{ \  \  \  \  \  \  \  \  \ }-\left.  \left(  1-k_{n+1}\right)  \mathbb{E}%
_{t_{n}}^{x}[\hat{F}_{t_{n+1+\delta_{n+1}^{-}}}+\int_{t_{n+1+\delta_{n+1}^{-}%
}}^{t_{n+1+\delta_{n+1}^{+}}}\frac{1}{2}\hat{F}_{r}^{(2)}dr+\int
_{t_{n+1+\delta_{n+1}^{-}}}^{t_{n+1+\delta_{n+1}^{+}}}\hat{F}_{r}^{(2)}%
dW_{r}]\right.  \\
\text{ \  \  \  \  \  \  \  \  \ }-\left.  k_{n}\mathbb{E}_{t_{n}}^{x}[F_{t_{n+\delta
_{n}^{-}}}]-k_{n+1}\mathbb{E}_{t_{n}}^{x}[\hat{F}_{t_{n+1+\delta_{n+1}^{-}}%
}]\right \vert \\
\text{ \  \  \  \  \  \ }\leq \frac{\Delta t}{2}\left \vert \int_{t_{n+\delta_{n}%
^{-}}}^{t_{n}+\delta(t_{n})}\mathbb{E}_{t_{n}}^{x}[F_{r}^{(2)}]dr-\left(
1-k_{n}\right)  \int_{t_{n+\delta_{n}^{-}}}^{t_{n+\delta_{n}^{+}}}%
\mathbb{E}_{t_{n}}^{x}[F_{r}^{(2)}]dr\right.  \\
\text{ \  \  \  \  \  \  \  \  \ }+\left.  \int_{t_{n+1+\delta_{n+1}^{-}}}%
^{t_{n+1}+\delta(t_{n+1})}\mathbb{E}_{t_{n}}^{x}[\hat{F}_{r}^{(2)}]dr-\left(
1-k_{n+1}\right)  \int_{t_{n+1+\delta_{n+1}^{-}}}^{t_{n+1+\delta_{n+1}^{+}}%
}\mathbb{E}_{t_{n}}^{x}[\hat{F}_{r}^{(2)}]dr\right \vert \\
\text{ \  \  \  \  \  \ }\leq C\left(  \Delta t\right)  ^{2},
\end{array}
\end{equation}
and
\[%
\begin{array}
[c]{l}%
|R_{z_{2}}^{n,\delta}|\leq \Delta t\left \vert \mathbb{E}_{t_{n}}^{x}[(\hat
{F}_{t_{n+1+\delta_{n+1}^{-}}}+\int_{t_{n+1+\delta_{n+1}^{-}}}^{t_{n+1}%
+\delta(t_{n+1})}\frac{1}{2}\hat{F}_{r}^{(2)}dr)\Delta W_{t_{n+1}}]\right.  \\
\text{ \  \  \  \  \  \  \  \  \ }+\left.  \mathbb{E}_{t_{n}}^{x}[\int_{t_{n+1+\delta
_{n+1}^{-}}}^{t_{n+1}+\delta(t_{n+1})}\hat{F}_{r}^{(2)}dW_{r}\Delta
W_{t_{n+1}}]-k_{n+1}\hat{F}_{t_{n+1+\delta_{n+1}^{-}}}\Delta W_{t_{n+1}%
}\right.  \\
\text{ \  \  \  \  \  \  \  \  \ }-\left.  \left(  1-k_{n+1}\right)  \mathbb{E}%
_{t_{n}}^{x}[(\int_{t_{n+1+\delta_{n+1}^{-}}}^{t_{n+1+\delta_{n+1}^{+}}}%
\frac{1}{2}\hat{F}_{r}^{(2)}dr-\int_{t_{n+1+\delta_{n+1}^{-}}}^{t_{n+1+\delta
_{n+1}^{+}}}\hat{F}_{r}^{(2)}dW_{r})\Delta W_{t_{n+1}}]\right.  \\
\text{ \  \  \  \  \  \  \  \  \ }-\left.  \left(  1-k_{n+1}\right)  \mathbb{E}%
_{t_{n}}^{x}[\hat{F}_{t_{n+1+\delta_{n+1}^{-}}}\Delta W_{t_{n+1}}]\right \vert
\\
\text{ \  \  \  \  \  \ }\leq \frac{\Delta t}{2}\left \vert \mathbb{E}_{t_{n}}%
^{x}[\int_{t_{n+1+\delta_{n+1}^{-}}}^{t_{n+1}+\delta(t_{n+1})}\hat{F}%
_{r}^{(2)}\Delta W_{t_{n+1}}]dr\right.  \\
\text{ \  \  \  \  \  \  \  \  \ }-\left.  \left(  1-k_{n+1}\right)  \mathbb{E}%
_{t_{n}}^{x}[\int_{t_{n+1+\delta_{n+1}^{-}}}^{t_{n+1+\delta_{n+1}^{+}}}\hat
{F}_{r}^{(2)}\Delta W_{t_{n+1}}dr]\right \vert \\
\text{ \  \  \  \  \  \ }\leq C\left(  \Delta t\right)  ^{\frac{5}{2}}.
\end{array}
\]
If $F\in C_{b}^{2,4,4}$, continue to apply It$\hat{o}$ formula to $F_{r}
^{(2)}$ and $\hat{F}_{r}^{(2)}$, using the fact $\delta(t_{n})-\delta_{n}
^{-}-\left(  1-k_{n}\right)  \Delta t=\delta(t_{n+1})-\delta_{n+1}^{-}-\left(
1-k_{n+1}\right)  \Delta t=0$, we can deduce
\[
\begin{array}
[c]{l}
|R_{y_{2}}^{n,\delta}|\leq\frac{\Delta t}{2}\left\vert \int_{t_{n+\delta
_{n}^{-}}}^{t_{n}+\delta(t_{n})}\mathbb{E}_{t_{n}}^{x}[F_{t_{n+\delta_{n}^{-}
}}^{(2)}+\int_{t_{n+\delta_{n}^{-}}}^{r}\frac{1}{2}F_{u}^{(4)}du]dr\right. \\
\text{ \ \ \ \ \ \ \ \ \ }-\int_{t_{n+\delta_{n}^{-}}}^{t_{n+\delta_{n}^{+}}
}\left(  1-k_{n}\right)  \mathbb{E}_{t_{n}}^{x}[F_{t_{n+\delta_{n}^{-}}}
^{(2)}+\int_{t_{n+\delta_{n}^{-}}}^{r}\frac{1}{2}F_{u}^{(4)}du]dr\\
\text{ \ \ \ \ \ \ \ \ \ }+\int_{t_{n+1+\delta_{n+1}^{-}}}^{t_{n+1}
+\delta(t_{n+1})}\mathbb{E}_{t_{n}}^{x}[\hat{F}_{t_{n+1+\delta_{n+1}^{-}}
}^{(2)}+\int_{t_{n+1+\delta_{n+1}^{-}}}^{r}\frac{1}{2}\hat{F}_{u}^{(4)}du]dr\\
\text{ \ \ \ \ \ \ \ \ \ }-\left.  \int_{t_{n+1+\delta_{n+1}^{-}}
}^{t_{n+1+\delta_{n+1}^{+}}}\left(  1-k_{n+1}\right)  \mathbb{E}_{t_{n}}
^{x}[\hat{F}_{t_{n+1+\delta_{n+1}^{-}}}^{(2)}+\int_{t_{n+1+\delta_{n+1}^{-}}
}^{r}\frac{1}{2}\hat{F}_{u}^{(4)}du]dr\right\vert \\
\text{ \ \ \ \ \ \ }\leq C\left(  \Delta t\right)  ^{3},
\end{array}
\]
and
\[
\begin{array}
[c]{l}
|R_{z_{2}}^{n,\delta}|\leq\frac{\Delta t}{2}\left\vert \int_{t_{n+1+\delta
_{n+1}^{-}}}^{t_{n+1}+\delta(t_{n+1})}\mathbb{E}_{t_{n}}^{x}[(\hat
{F}_{t_{n+1+\delta_{n+1}^{-}}}^{(2)}+\int_{t_{n+1+\delta_{n+1}^{-}}}^{r}
\frac{1}{2}\hat{F}_{u}^{(4)}du)\Delta W_{t_{n+1}}]dr\right. \\
\text{ \ \ \ \ \ \ \ \ \ \ }-\left.  \left(  1-k_{n+1}\right)  \int
_{t_{n+1+\delta_{n+1}^{-}}}^{t_{n+1+\delta_{n+1}^{+}}}\mathbb{E}_{t_{n}}
^{x}[(\hat{F}_{t_{n+1+\delta_{n+1}^{-}}}^{(2)}+\int_{t_{n+1+\delta_{n+1}^{-}}
}^{r}\frac{1}{2}\hat{F}_{u}^{(4)}du)\Delta W_{t_{n+1}}]dr\right\vert \\
\text{ \ \ \ \ \ \ }\leq C\left(  \Delta t\right)  ^{\frac{7}{2}}.
\end{array}
\]
We can also verify that $|R_{y_{3}}^{n,\delta}|\leq C\left(  \Delta t\right)
^{3}$. To sum up, we get the estimates of $|R_{y}^{n,\delta}|$ and
$|R_{z}^{n,\delta}|$. The proof is complete.
\end{proof}
\end{lemma}

Based on the above discussion, now let us show the error estimates of
$Y_{t_{n}}-Y^{n}$ and $Z_{t_{n}}-Z^{n}$.

\begin{theorem}
\label{Theorem 4.2} Let $\left(  Y_{t},Z_{t}\right)  ,t\in\left[
0,T+S\right]  $, be the solution of the anticipated BSDEs $\left(
\ref{1.2}\right)  $ and $\left(  Y^{n},Z^{n}\right)  (  n=N+K,\ldots
,0)  $ defined in the explicit scheme \ref{scheme}. Suppose that $f$
satisfies the regularity conditions with Lipschitz constant $L$, $\sqrt
{\delta(t)}$ is smooth enough, $\hat{T}=\left(  N+K\right)  \Delta t=T+S$.
Then for sufficiently small time step $\Delta t$, we have the following estimates.

\begin{enumerate}
\item For $\theta_{i}\left(  i=1,2\right)  \in\lbrack0,1]$, $\theta_{3}
\in(0,1]$, if $\varphi\in C_{b}^{\frac{1}{2},2}$, $f\in C_{b}^{\frac{1}
{2},1,1,1,1}$, $M_{ey}=O\left(  \Delta t\right)  $ and $M_{ez}=O\left(  \Delta
t\right)  $, we have
\begin{equation}
\max_{0\leq n\leq N-1}\left(  \mathbb{E[}|e_{y}^{n}|^{2}]+\Delta
t\sum\limits_{i=n}^{N-1}\mathbb{E[}\left\vert e_{z}^{n}\right\vert
^{2}]\right)  \leq C\Delta t. \label{4.38}
\end{equation}

and if $\varphi\in C_{b}^{1,4}$, $f\in C_{b}^{1,3,3,3,3}$, $M_{ey}=O\left(
\Delta t^{2}\right)  $ and $M_{ez}=O\left(  \Delta t^{2}\right)  $, we have
\begin{equation}
\max_{0\leq n\leq N-1}\left(  \mathbb{E[}|e_{y}^{n}|^{2}]+\Delta
t\sum\limits_{i=n}^{N-1}\mathbb{E[}\left\vert e_{z}^{n}\right\vert
^{2}]\right)  \leq C\left(  \Delta t\right)  ^{2}.\ \label{4.39}
\end{equation}

\item In particular, for $\theta_{i}=\frac{1}{2}\left(  i=1,2,3\right)  $, if
$\varphi\in C_{b}^{2,6}$, $f\in C_{b}^{2,5,5,5,5}$, $M_{ey}=O\left(  \Delta
t^{4}\right)  $ and $M_{ez}=O\left(  \Delta t^{4}\right)  $, we have
\begin{equation}
\max_{0\leq n\leq N-1}\left(  \mathbb{E[}|e_{y}^{n}|^{2}]+\Delta
t\sum\limits_{i=n}^{N-1}\mathbb{E[}\left\vert e_{z}^{n}\right\vert
^{2}]\right)  \leq C\left(  \Delta t\right)  ^{4}.\ \label{4.40}
\end{equation}

\end{enumerate}
Here $C$ is a constant depending on $L$, $\hat{T}$, $\theta_{i}\left(
i=1,2,3\right)  $, and the upper bounds of derivatives of $\varphi$,
$f$.

\begin{proof}
By Lemma \ref{truncation error lemma}, for sufficiently time step $\Delta t$,
it holds that

\begin{enumerate}
\item For $\theta_{i}\left(  i=1,2\right)  \in\lbrack0,1]$, $\theta_{3}
\in(0,1]$, if $\varphi\in C_{b}^{\frac{1}{2},2}$ and $f\in C_{b}^{\frac{1}
{2},1,1,1,1}$, we have
\[
\left\vert R_{y}^{n,\delta}\right\vert \leq C(\Delta t)^{\frac{3}{2}
},\;\left\vert R_{z}^{n,\delta}\right\vert \leq C(\Delta t)^{\frac{3}{2}},
\]

and if $\varphi\in C_{b}^{1,4}$ and$\ f\in C_{b}^{1,3,3,3,3}$, we have
\[
\left\vert R_{y}^{n,\delta}\right\vert \leq C(\Delta t)^{2},\;\left\vert
R_{z}^{n,\delta}\right\vert \leq C(\Delta t)^{2},
\]

\item In particular, for $\theta_{i}=\frac{1}{2}\left(  i=1,2,3\right)  $, if
$\varphi\in C_{b}^{2,6}$ and $f\in C_{b}^{2,5,5,5,5}$, we have
\[
\left\vert R_{y}^{n,\delta}\right\vert \leq C(\Delta t)^{3},\;\left\vert
R_{z}^{n,\delta}\right\vert \leq C(\Delta t)^{3}.
\]

\end{enumerate}
Using Theorem \ref{Theorem 4.1}, the conclusion can be directly obtained.
\end{proof}
\end{theorem}

\begin{remark}
The Theorem \ref{Theorem 4.2} above implies that our explicit $\theta$-schemes
could have high order as well as high accuracy for solving anticipated BSDEs.
When $f$, $\varphi$ and $\delta(t)$ are smooth enough the convergence order of
schemes \ref{scheme} for solving $Y$ and $Z$ is 1. In particular, if the
parameters $\theta_{1}=\theta_{2}=\theta_{3}=\frac{1}{2}$, the convergence
rate can reach second order.
\end{remark}

\section{Numerical experiments\label{sec5}}

In this section, we will give some numerical experiments to illustrate the
high accuracy of our explicit $\theta$-schemes for solving the anticipated
BSDEs $\left(  \ref{1.2}\right)  $.

We first give the numerical approximation methods for the conditional
mathematical expectation which includes the delay process. To simplify the
process, we just give the approximation of $\mathbb{E}_{t_{n}}^{x}[f\left(
t_{n+1},Y^{n+1},Y^{n+m+1}\right)  ]$ $(m=\delta_{n+1}^{-}$, $\delta_{n+1}
^{+})$ with the generator $f$ only depends on $Y$
\[
\begin{array}
[c]{l}
\mathbb{E}_{t_{n}}^{x}[f\left(  t_{n+1},Y^{n+1},Y^{n+m+1}\right)  ]\\
=\mathbb{E}_{t_{n}}^{x}\left[  \mathbb{E}_{t_{n+1}}\left[  f\left(
t_{n+1},Y^{n+1}(W_{t_{n+1}}),Y^{n+m+1}(W_{t_{n+K+1}})\right)  \right]  \right]
\\
=\mathbb{E}_{t_{n}}^{x}\left[  \left.  \mathbb{E[}f\left(  t_{n+1}
,Y^{n+1}(p),Y^{n+m+1}(W_{t_{n+m+1}}-W_{t_{n+1}}+q)\right)  ]\right\vert
_{p=q=W_{t_{n+1}}}\right] \\
=\frac{1}{\sqrt{2\pi}}\mathbb{E}\left[  \int_{\mathbb{R}}f\left(
t_{n+1},Y^{n+1}(x+\Delta W_{t_{n+1}}),\right.  \right. \\
\text{ \ \ \ \ \ \ \ \ \ \ \ \ \ \ }\left.  \left.  Y^{n+m+1}(\sqrt{m\Delta
t}\xi+x+\Delta W_{t_{n+1}})\right)  e^{-\frac{\;\xi^{2}}{2}}d\xi\right] \\
=\frac{1}{\pi}\int_{\mathbb{R}}\int_{\mathbb{R}}f\left(  t_{n+1}
,Y^{n+1}(x+\sqrt{2\Delta t}\eta),\right. \\
\text{ \ \ \ \ \ \ \ \ \ \ \ \ \ \ }\left.  Y^{n+m+1}(x+\sqrt{2\Delta t}
\eta+\sqrt{2m\Delta t}\gamma)\right)  e^{-(\gamma^{2}+\eta^{2})}d\gamma
d\eta\\
\approx\frac{1}{\pi}\sum\limits_{i=1}^{G_{1}}\sum\limits_{j=1}^{G_{2}}
w_{i}w_{j}f\left(  t_{n+1},Y^{n+1}(x+\sqrt{2\Delta t}p_{i}),\right. \\
\text{ \ \ \ \ \ \ \ \ \ \ \ \ \ \ }\left.  Y^{n+m+1}(x+\sqrt{2\Delta t}
p_{i}+\sqrt{2m\Delta t}q_{j})\right)  .
\end{array}
\]
where $\left\{  w_{i}\right\}  _{i=1}^{G_{1}}$ and $\left\{  w_{j}\right\}
_{j=1}^{G_{2}}$ are the weights of Gauss-Hermite quadrature formula, $\left\{
p_{i}\right\}  _{i=1}^{G_{1}}$ and $\left\{  q_{j}\right\}  _{j=1}^{G_{2}}$
are the roots of the Hermite polynomials of degree $G_{1}$ and $G_{2}$,
respectively. The conditional mathematical expectations $\mathbb{E}_{t_{n}
}^{x}[Y^{n+1}]$, $\mathbb{E}_{t_{n}}^{x}\left[  \bar{f}_{-}^{n}\right]  $,
$\mathbb{E}_{t_{n}}^{x}\left[  \bar{f}_{+}^{n}\right]  $, $\mathbb{E}_{t_{n}
}^{x}[Y^{n+1}\Delta W_{t_{n+1}}^{\intercal}]$, $\mathbb{E}_{t_{n}}^{x}
[f_{-}^{n+1}\Delta W_{t_{n+1}}^{\intercal}]$ and $\mathbb{E}_{t_{n}}^{x}
[f_{+}^{n+1}\Delta W_{t_{n+1}}^{\intercal}]$ can be obtained similarly, see
more details in \cite{ZhWP2009,ZhZhJ2010}.

Now let us take a few notes on our numerical experiments. We take a uniform
partition of time and space, $\Delta t$ and $\Delta x$ represent time and
space steps, respectively. In our tests, we use Gauss-Hermite quadrature rule
to approximate the conditional mathematical expectations and apply cubic
spline interpolation to compute spatial non-grid points. In order to balance
the time and space error, we set $\Delta t$ and $\Delta x$ satisfy the
equality $(\Delta t)^{p+1}=(\Delta x)^{q}$, where $q=4$ is the order of
spatial interpolation error and $p$ is the convergence order of the numerical
scheme. In the following tables, CR stands for convergence rate with respect
to time step $\Delta t$, $|Y_{0}-Y^{0}|$ and $|Z_{0}-Z^{0}|$ denote the
absolute errors between the exact and numerical solutions for $Y_{t}$ and
$Z_{t}$ at $(t_{0},x_{0})$, respectively.

\begin{example}\label{example1}
In this example, we consider a linear anticipated BSDEs
\begin{equation*}
-dY_{t}=\mathbb{E[-}\frac{1}{2}Y_{t}-(Y_{t+\frac{\text{ }t^{2}}{4}}\sin
\frac{\text{ }t^{2}}{4}+Z_{t+\frac{\text{ }t^{2}}{4}}\cos\frac{\text{ }t^{2}
}{4})e^{\frac{-t^{2}}{8}}|\mathcal{F}_{t}]dt-Z_{t}dW_{t},\text{ \ }t\in
\lbrack0,T],
\end{equation*}
with the terminal condition
\[Y_{t}=\exp\left(  t\right)  \sin\left(  t+W_{t}\right)  ,\;\;Z_{t}=\exp\left(
t\right)  \cos\left(  t+W_{t}\right)  ,\text{ \ }t\in\left[  T,T+S\right]  .\]

The exact solution is $(Y_{t},Z_{t})=(\exp(t)\sin(t+W_{t}),\exp(t)\cos
(t+W_{t}))$. We let the Brownian motion start at the time-space point $(0,0)$
and $T=1$, then $S=\frac{1}{4}$, the exact solution $\left(  Y_{t}
,Z_{t}\right)  $ at initial time $t_{0}=0$ is $\left(  Y_{0},Z_{0}\right)
=\left(  0,1\right)  $. In table \ref{Table1}, we have listed the errors
$\left\vert Y_{0}-Y^{0}\right\vert $ and $\left\vert Z_{0}-Z^{0}\right\vert $
and their convergence rates of our explicit $\theta$-schemes with different
time partitions.
\begin{table}[ptbh]
\caption{Errors and convergence rates of Example \ref{example1}.}
\label{Table1}
{\footnotesize
\begin{align*}
&
\begin{tabular}
[c]{l|c|c|c|c}\cline{1-2}\cline{2-5}
& \multicolumn{2}{|c}{$\theta_{1}=0.5$, $\theta_{2}=0.5$, $\theta_{3}=0.5$} &
\multicolumn{2}{|c}{$\theta_{1}=1,\ \theta_{2}=0$, $\theta_{3}$ $=0.5$
}\\\hline
$M$ & $\ \ \left\vert Y_{0}-Y^{0}\right\vert $\ \ \  & $\left\vert Z_{0}
-Z^{0}\right\vert $ & $\ \left\vert Y_{0}-Y^{0}\right\vert $ \  & $\left\vert
Z_{0}-Z^{0}\right\vert $\\\hline
35 & 2.329E-04 & 2.575E-04 & 2.844E-02 & 6.077E-02\\\hline
55 & 9.520E-05 & 1.026E-04 & 1.812E-02 & 3.921E-02\\\hline
75 & 5.068E-05 & 5.143E-05 & 1.330E-02 & 2.895E-02\\\hline
95 & 3.201E-05 & 3.484E-05 & 1.050E-02 & 2.288E-02\\\hline
115 & 2.195E-05 & 2.407E-05 & 8.675E-03 & 1.895E-02\\\hline
CR & 1.988 & 2.000 & 0.998 & 0.980\\\cline{1-2}\cline{2-5}\cline{4-4}
\end{tabular}
\\
&
\begin{tabular}
[c]{l|c|c|c|c}\cline{1-2}\cline{2-5}
& \multicolumn{2}{|c}{$\theta_{1}=0$,$\ \theta_{2}=0.5$,$\ \theta_{3}$ $=0.5$}
& \multicolumn{2}{|c}{$\theta_{1}=0$, $\theta_{2}=0.25$, $\theta_{3}$ $=0.5$
}\\\hline
$M$ & $\ \ \left\vert Y_{0}-Y^{0}\right\vert $\ \  & $\left\vert Z_{0}
-Z^{0}\right\vert $ & $\ \left\vert Y_{0}-Y^{0}\right\vert $\ \  & $\left\vert
Z_{0}-Z^{0}\right\vert $\\\hline
\multicolumn{1}{l|}{35} & 1.690E-02 & 1.494E-02 & 1.170E-02 &
9.615E-03\\\hline
\multicolumn{1}{l|}{55} & 1.076E-02 & 9.367E-03 & 7.307E-03 &
6.223E-03\\\hline
\multicolumn{1}{l|}{75} & 7.897E-03 & 6.820E-03 & 5.310E-03 &
4.600E-03\\\hline
\multicolumn{1}{l|}{95} & 6.235E-03 & 5.362E-03 & 4.175E-03 &
3.614E-03\\\hline
\multicolumn{1}{l|}{115} & 5.152E-03 & 4.418E-03 & 3.437E-03 &
2.995E-03\\\hline
CR & 0.999 & 1.025 & 1.030 & 0.982\\\cline{1-2}\cline{2-5}\cline{4-4}
\end{tabular}
\\
&
\begin{tabular}
[c]{l|c|c|c|c}\cline{1-2}\cline{2-5}
& \multicolumn{2}{|c}{$\theta_{1}=1$, $\theta_{2}=0.5$,$\theta_{3}$ $=0.75$} &
\multicolumn{2}{|c}{$\theta_{1}=0.5$,$\theta_{2}=0.5$,$\theta_{3}$ $=0.25$
}\\\hline
$M$ & $\ \ \left\vert Y_{0}-Y^{0}\right\vert $\ \  & $\left\vert Z_{0}
-Z^{0}\right\vert $ & $\ \left\vert Y_{0}-Y^{0}\right\vert $\ \  & $\left\vert
Z_{0}-Z^{0}\right\vert $\\\hline
35 & 1.290E-02 & 1.476E-03 & 1.108E-02 & 4.766E-02\\\hline
55 & 8.239E-03 & 1.003E-03 & 7.180E-03 & 3.058E-02\\\hline
75 & 6.053E-03 & 7.519E-04 & 5.313E-03 & 2.253E-02\\\hline
95 & 4.786E-03 & 5.878E-04 & 4.204E-03 & 1.776E-02\\\hline
115 & 3.956E-03 & 4.921E-04 & 3.486E-03 & 1.470E-02\\\hline
CR & 0.994 & 0.928 & 0.972 & 0.989\\\cline{1-2}\cline{2-5}\cline{4-4}
\end{tabular}
\end{align*}
}\end{table}
\end{example}

\begin{table}[ptbh]
\caption{Errors and convergence rates of Example \ref{example2}.}
\label{Table2}
{\footnotesize
\begin{align*}
&
\begin{tabular}
[c]{l|c|c|c|c}\cline{1-2}\cline{2-5}
& \multicolumn{2}{|c}{$\theta_{1}=0.5$, $\theta_{2}=0.5$, $\theta_{3}=0.5$} &
\multicolumn{2}{|c}{$\theta_{1}=1$,$\ \theta_{2}=1$,$\ \theta_{3}$ $=1$
}\\\hline
$M$ & $\ \ \left\vert Y_{0}-Y^{0}\right\vert $\ \ \  & $\left\vert Z_{0}
-Z^{0}\right\vert $ & $\ \left\vert Y_{0}-Y^{0}\right\vert $\ \  & $\left\vert
Z_{0}-Z^{0}\right\vert $\\\hline
35 & 1.980E-04 & 4.511E-04 & 1.047E-02 & 9.548E-03\\\hline
55 & 7.953E-05 & 1.828E-04 & 6.729E-03 & 6.103E-03\\\hline
75 & 4.602E-05 & 9.998E-05 & 4.949E-03 & 4.481E-03\\\hline
95 & 2.573E-05 & 6.076E-05 & 4.001E-03 & 3.576E-03\\\hline
115 & 1.734E-05 & 4.118E-05 & 3.306E-03 & 2.956E-03\\\hline
CR & 2.040 & 2.010 & 0.967 & 0.985\\\cline{1-2}\cline{2-5}\cline{4-4}
\end{tabular}
\\
&
\begin{tabular}
[c]{l|c|c|c|c}\cline{1-2}\cline{2-5}
& \multicolumn{2}{|c}{$\theta_{1}=0$, $\theta_{2}=0.5$, $\theta_{3}$ $=0.5$} &
\multicolumn{2}{|c}{$\theta_{1}=0$, $\theta_{2}=0.25$, $\theta_{3}$ $=0.5$
}\\\hline
$M$ & $\ \ \left\vert Y_{0}-Y^{0}\right\vert $\ \  & $\left\vert Z_{0}
-Z^{0}\right\vert $ & $\ \ \left\vert Y_{0}-Y^{0}\right\vert $\ \  &
$\left\vert Z_{0}-Z^{0}\right\vert $\\\hline
35 & 2.309E-02 & 8.741E-03 & 1.736E-02 & 1.002E-02\\\hline
55 & 1.477E-02 & 5.736E-03 & 1.097E-02 & 6.270E-03\\\hline
75 & 1.085E-02 & 4.267E-03 & 8.020E-03 & 4.559E-03\\\hline
95 & 8.579E-03 & 3.397E-03 & 6.358E-03 & 3.596E-03\\\hline
115 & 7.093E-03 & 2.821E-03 & 5.244E-03 & 2.959E-03\\\hline
CR & 0.992 & 0.951 & 1.006 & 1.025\\\cline{1-2}\cline{2-5}\cline{4-4}
\end{tabular}
\\
&
\begin{tabular}
[c]{l|c|c|c|c}\cline{1-2}\cline{2-5}
& \multicolumn{2}{|c}{$\theta_{1}=1$, $\theta_{2}=0.5$, $\theta_{3}$ $=0.75$}
& \multicolumn{2}{|c}{$\theta_{1}=1,\ \theta_{2}=0$, $\theta_{3}$ $=0.5$
}\\\hline
$M$ & $\ \ \left\vert Y_{0}-Y^{0}\right\vert $\ \  & $\left\vert Z_{0}
-Z^{0}\right\vert $ & $\ \ \left\vert Y_{0}-Y^{0}\right\vert $ \  &
$\left\vert Z_{0}-Z^{0}\right\vert $\\\hline
35 & 1.881E-02 & 8.935E-03 & 3.430E-02 & 9.079E-03\\\hline
55 & 1.207E-02 & 5.859E-03 & 2.226E-02 & 5.904E-03\\\hline
75 & 8.880E-03 & 4.351E-03 & 1.647E-02 & 4.369E-03\\\hline
95 & 7.058E-03 & 3.475E-03 & 1.300E-02 & 3.439E-03\\\hline
115 & 5.837E-03 & 2.884E-03 & 1.078E-02 & 2.854E-03\\\hline
CR & 0.983 & 0.951 & 0.974 & 0.974\\\cline{1-2}\cline{2-5}\cline{4-4}
\end{tabular}
\end{align*}
}\end{table}

\begin{example}\label{example2}
In this example, we test the following nonlinear anticipated
BSDEs with $\delta(t)=\frac{1}{4}t^{2}$.
\begin{equation*}
\left\{
\begin{array}
[c]{ll}
-dY_{t}=\mathbb{E[}\frac{Y_{t}-2\left(  Y_{t+\delta(t)}\sin\delta
(t)+Z_{t+\delta(t)}\cos\delta(t)\right)  e^{\frac{\delta(t)}{2}}}{Y_{t}
^{2}+Z_{t}^{2}+1}|\mathcal{F}_{t}]dt-Z_{t}dW_{t}, & t\in\lbrack0,T];\\
Y_{t}=\sin\left(  t+W_{t}\right)  , & t\in\left[  T,T+S\right]  ;\\
Z_{t}=\cos\left(  t+W_{t}\right)  , & t\in\left[  T,T+S\right]  .
\end{array}
\right.
\end{equation*}
The exact solution is
\[
\begin{array}
[c]{ll}
Y_{t}=\sin\left(  t+W_{t}\right)  , & \;Z_{t}=\cos\left(  t+W_{t}\right)
,\text{ \ \ \ \ \ }t\in\lbrack0,T+S].
\end{array}
\]
The Brownian motion starts at the time-space point $(0,0),$ $T=1$ and
$S=\frac{1}{4}$, the exact solution $\left(  Y_{t},Z_{t}\right)  $ at initial
time $t_{0}=0$ is $\left(  Y_{0},Z_{0}\right)  =\left(  0,1\right)  $. In
table \ref{Table2}, we show the errors $\left\vert Y_{0}-Y^{0}\right\vert $
and $\left\vert Z_{0}-Z^{0}\right\vert $ and their convergence rates of our
explicit $\theta$-schemes with different time partitions.
\end{example}

\begin{remark}
The numerical results above show that our numerical schemes defined in
$\left(  \ref{scheme}\right)  $ work very well for solving anticipated BSDEs
whether the generator is linear or nonlinear.
\end{remark}

\begin{remark}
The convergence rate of the scheme depends on the choice of parameters. If
$\theta_{i}\left(  i=1,2,3\right)  $ is not equal to $\frac{1}{2}$, the
convergence rate of the schemes is first order, if $\theta_{1}=\theta
_{2}=\theta_{3}=\frac{1}{2}$, the convergence rate is second order, which is
in good agreement with our theoretical analysis.
\end{remark}
\section{Conclusions}

In this paper, we propose a class of stable explicit $\theta$-schemes for
solving anticipated BSDEs. By choosing different $\theta_{i}$ $\left(
i=1,2,3\right)  $, we obtain many different numerical schemes. We also analyze
the stability of our numerical schemes and strictly prove the error estimates.
When $\theta_{1}=\theta_{2}=\theta_{3}=\frac{1}{2}$, the scheme is a second
order numerical scheme, which is very effective for solving anticipated BSDEs.
The numerical tests in section 5 powerful back up the theoretical results.
This seems to be the first time to design high order numerical schemes for
anticipated BSDEs.

\appendix

\section{Appendix}

\subsection{The proof of Lemma \ref{lemma}}

\begin{proof}
For $\left(  \ref{4.1}\right)  $, replacing $n$ by $n+s\ $with $0\leq s\leq
N-n-1$ and multiplying $\left(  1+C_{2}\Delta t\right)  ^{s}$ on the both side
of the inequality, we rewrite it as follow%

\begin{equation}%
\begin{array}
[c]{l}%
(1+C_{2}\Delta t)^{s}A_{n+s}+(1+C_{2}\Delta t)^{s}C_{1}\Delta tB_{n+s}\\
\qquad\qquad\ \ \ \ \ \ \ \ \ \ \ \ \leq(1+C_{2}\Delta t)^{s+1}A_{n+s+1}%
+(1+C_{2}\Delta t)^{s}C_{3}\Delta t\sum\limits_{i\in\mathcal{I}_{n+s}%
}A_{n+s+i}\\
\qquad\qquad\quad\ \ \ \ \ \ \ \ \ \ \ \ +(1+C_{2}\Delta t)^{s}C_{3}\Delta
t\sum_{i\in\mathcal{I}_{n+s}}B_{n+s+i}+(1+C_{2}\Delta t)^{s}R^{n+s}.
\end{array}
\label{4.3}%
\end{equation}
Adding up the both side of $\left(  \ref{4.3}\right)  $ about $s$ from $0$ to
$N-n-1$, we get%
\begin{equation}%
\begin{array}
[c]{l}%
A_{n}+C_{1}\Delta t\sum\limits_{s=0}^{N-n-1}(1+C_{2}\Delta t)^{s}B_{n+s}\\
\;\ \ \ \ \ \leq\left(  1+C_{2}\Delta t\right)  ^{N-n}A_{N}+C_{3}\Delta
t\sum\limits_{s=0}^{N-n-1}(1+C_{2}\Delta t)^{s}\sum\limits_{i\in
\mathcal{I}_{n+s}}A_{n+s+i}\\
\;\ \ \ \ \ \ \ \ +C_{3}\Delta t\sum\limits_{s=0}^{N-n-1}(1+C_{2}\Delta
t)^{s}\sum_{i\in\mathcal{I}_{n+s}}B_{n+s+i}+\sum\limits_{s=0}^{N-n-1}%
(1+C_{2}\Delta t)^{s}R^{n+s}.
\end{array}
\label{4.4}%
\end{equation}
Observe that fixed $i\in\mathcal{I}_{n+s}$, we have
\[%
\begin{array}
[c]{l}%
\sum\limits_{s=0}^{N-n-1}(1+C_{2}\Delta t)^{s}A_{n+s+i}\leq(1+C_{2}\Delta
t)^{N-n-1}\sum\limits_{i=n+1}^{N+K}A_{i},\\
\sum\limits_{s=0}^{N-n-1}(1+C_{2}\Delta t)^{s}B_{n+s+i}\leq\sum\limits_{i=n+1}%
^{N+K}(1+C_{2}\Delta t)^{i-n}B_{i},
\end{array}
\]
then
\begin{equation}%
\begin{array}
[c]{l}%
\sum\limits_{s=0}^{N-n-1}(1+C_{2}\Delta t)^{s}\sum\limits_{i\in\mathcal{I}%
_{n+s}}A_{n+s+i}\leq I_{0}(1+C_{2}\Delta t)^{N-n-1}\sum\limits_{i=n+1}%
^{N+K}A_{i},\\
\sum\limits_{s=0}^{N-n-1}(1+C_{2}\Delta t)^{s}\sum\limits_{i\in\mathcal{I}%
_{n+s}}B_{n+s+i}\leq I_{0}\sum\limits_{i=n+1}^{N+K}(1+C_{2}\Delta
t)^{i-n}B_{i}.
\end{array}
\label{4.5}%
\end{equation}
\bigskip Inserting $\left(  \ref{4.5}\right)  $ into $\left(  \ref{4.4}%
\right)  $, we deduce that:
\begin{equation}%
\begin{array}
[c]{l}%
A_{n}+C_{1}\Delta t\sum\limits_{i=n}^{N-1}(1+C_{2}\Delta t)^{i-n}B_{i}\\
\;\ \ \ \ \ \ \ \leq e^{C_{2}\hat{T}}A_{N}+I_{0}C_{3}e^{C_{2}\hat{T}}\Delta
t\sum\limits_{i=n+1}^{N+K}A_{i}\ +I_{0}C_{3}\Delta t\sum\limits_{i=n}%
^{N+K}(1+C_{2}\Delta t)^{i-n}B_{i}\\
\;\ \ \ \ \ \ \ \ \ \ +\sum\limits_{i=0}^{N-1}(1+C_{2}\Delta t)^{i}R^{i}%
\end{array}
\label{4.6}%
\end{equation}
Which leads to%
\begin{equation}%
\begin{array}
[c]{l}%
A_{n}+\left(  C_{1}-I_{0}C_{3}\right)  \Delta t\sum\limits_{i=n}^{N-1}%
(1+C_{2}\Delta t)^{i-n}B_{i}\\
\;\ \ \ \ \ \ \ \leq e^{C_{2}\hat{T}}A_{N}+I_{0}C_{3}e^{C_{2}\hat{T}}\Delta
t\sum\limits_{i=n+1}^{N+K}A_{i}+I_{0}C_{3}\Delta t\sum\limits_{i=N}%
^{N+K}(1+C_{2}\Delta t)^{i}B_{i}\\
\;\ \ \ \ \ \ \ \ \ \ +\sum\limits_{i=0}^{N-1}(1+C_{2}\Delta t)^{i}R^{i}.
\end{array}
\label{4.7}%
\end{equation}
Let
\begin{align*}
\alpha &  =I_{0}C_{3}e^{C_{2}\hat{T}},\\
\beta &  =e^{C_{2}\hat{T}}A_{N}++I_{0}C_{3}\Delta t\sum\limits_{i=N}%
^{N+K}(1+C_{2}\Delta t)^{i}B_{i}+\sum\limits_{i=0}^{N-1}(1+C_{2}\Delta
t)^{i}R^{i},
\end{align*}
From $\left(  \ref{4.7}\right)  $, we have
\begin{equation}
A_{n}\leq\alpha\Delta t\sum\limits_{i=n+1}^{N+K}A_{i}+\beta\label{4.8}%
\end{equation}
Using Lemma $\ref{gronwall lemma}$ to $\left(  \ref{4.8}\right)  $ we obtain
\[
A_{n}\leq e^{\alpha\hat{T}}\left(  \beta+\alpha K\Delta tM_{A}\right)
,\;\;M_{A}=\max_{N\leq i\leq N+K}A_{i}.
\]
which is%
\begin{equation}%
\begin{array}
[c]{l}%
A_{n}\leq e^{\alpha\hat{T}}\left(  e^{C_{2}\hat{T}}A_{N}+I_{0}C_{3}\Delta
t\sum\limits_{i=N}^{N+K}(1+C_{2}\Delta t)^{i}B_{i}\right. \\
\;\ \ \ \ \ \ \left.  +\sum\limits_{i=0}^{N-1}(1+C_{2}\Delta t)^{i}%
R^{i}+\alpha K\Delta tM_{A}\right)  .
\end{array}
\label{4.9}%
\end{equation}
Let $C_{5}=$ $e^{\alpha\hat{T}}\max\left(  e^{C_{2}\hat{T}},I_{0}%
C_{3},1,\alpha\right)  $, Inserting $\left(  \ref{4.9}\right)  $ into the
right side of $\left(  \ref{4.7}\right)  $, since $C_{1}-I_{0}C_{3}\geq
C_{4}>0$, we get%
\begin{equation}%
\begin{array}
[c]{l}%
A_{n}+C_{4}\Delta t\sum\limits_{i=n}^{N-1}B_{i}\\
\;\ \ \ \ \ \leq e^{C_{2}\hat{T}}A_{N}+I_{0}C_{3}\Delta t\sum\limits_{i=N}%
^{N+K}(1+C_{2}\Delta t)^{i}B_{i}+\sum\limits_{i=0}^{N-1}(1+C_{2}\Delta
t)^{i}R^{i}\\
\;\ \ \ \ \ \ \ \ +\alpha\Delta t\sum\limits_{i=n+1}^{N+K}\left[  C_{5}\left(
A_{N}+\Delta t\sum\limits_{j=N}^{N+K}(1+C_{2}\Delta t)^{j}B_{j}+K\Delta
tM_{A}\right.  \right. \\
\;\ \ \ \ \ \ \ \ +\left.  \left.  \sum\limits_{i=0}^{N-1}(1+C_{2}\Delta
t)^{i}R^{i}\right)  \right] \\
\;\ \ \ \ \ \leq\left(  e^{C_{2}\hat{T}}+\alpha\hat{T}C_{5}\right)
A_{N}+e^{C_{2}\hat{T}}\left(  3C_{3}\hat{T}+\alpha C_{5}\hat{T}^{2}\right)
M_{B}\\
\;\ \ \ \ \ \ \ \ +\alpha C_{5}\hat{T}^{2}M_{A}+e^{C_{2}\hat{T}}\left(
1+\alpha\hat{T}C_{5}\right)  \sum\limits_{i=0}^{N-1}R^{i}\\
\;\ \ \ \ \ \leq C_{6}\left(  A_{N}+M_{A}+M_{B}+\sum\limits_{i=0}^{N-1}%
R^{i}\right)  ,
\end{array}
\label{4.10}%
\end{equation}
where $C_{6}=2e^{C_{2}\hat{T}}\max\left\{  1,I_{0}C_{3}\hat{T},\alpha
C_{5}\hat{T},\alpha C_{5}\hat{T}^{2}\right\}  $. We can complete the proof by
using $C=\frac{C_{6}}{\min\left\{  1,C_{4}\right\}  }$ in $\left(
\ref{4.10}\right)  $.
\end{proof}


\begin{thebibliography}{99}                                                                                            %
\bibitem {B1997}
{\sc V. Bally},
{\em Approximation scheme for solutions of BSDE},
in Backward Stochastic Dfferential Equations, Pitman Res. Notes Math. 364, Longman, Harlow, UK, 1997, pp. 177-191.

\bibitem {BCY2015}
{\sc A. Bensoussan, M. H. M. Chau, and S. C. P. Yam},
{\em Mean field Stackelberg games: Aggregation of delayed instructions},
SIAM J. Control Optim., 53 (2015), pp. 2237-2266.

\bibitem {BCY2016}
{\sc A. Bensoussan, M. H. M. Chau, and S. C. P. Yam},
{\em Mean field games with a dominating player},
Appl Math Optim., 74 (2016), pp. 91-128.

\bibitem {BD2007}
{\sc C. Bender and R. Denk},
{\em A forward scheme for backward SDEs},
Stochastic Process. Appl., 117 (2007), pp. 1793-1812.

\bibitem {C2013}
{\sc J. F. Chassagneux},
{\em Linear Multistep Schemes for BSDEs},
SIAM J. Numer. Anal. 52 (2014), pp. 2815-2836.

\bibitem {CW2010}
{\sc L. Chen and Z. Wu},
{\em Maximum principle for the stochastic optimal control problem with delay and application},
Automatica, 46 (2010), pp. 1074-1080.

\bibitem {C1997}
{\sc D. Chevance},
{\em Numerical methods for backward stochastic differential equations},
in Numerical Methods in Finance, Edited by L. C. G. Rogers and D. Talay, Cambridge University Press, 1997, pp. 232-244.


\bibitem {CM}
{\sc D. Crisan and K. Manolarakis},
{\em Second order discretization of Backward SDEs},
Ann. Appl. Probab., 24 (2014), pp. 652-678.

\bibitem {GL2010}
{\sc E. Gobet and C. Labart},
{\em Solving BSDE with adaptive control variate},
SIAM Numer. Anal. 48 (2010), pp. 257-277.

\bibitem {GLW2005}
{\sc E. Gobet, J.-P. Lemor, and X. Warin},
{\em A regression-based Monte Carlo method to solve backward stochastic differential equations},
Ann. Appl. Probab. 15 (2005), pp. 2172-2202.

\bibitem {KPQ1997}
{\sc N. EI Karoui, S. Peng, and M.C., Quenez}
{\em Backward stochastic differential equations in finance},
Math. Finance, 7 (1997), pp. 1-71.

\bibitem {LSU1968}
{\sc O. A. Lady$\check{z}$enskaja, V. A. Solonnikov and N. N. Ural'ceva},
{\em Linear and Quasi-Linear Equations of Parabolic Type},
Translations of Math. Monographs, American Mathematical Society, 23, Providence, R.I., 1968.

\bibitem {MP2002}
{\sc J. Ma, P. Protter, J. San Martin, and S. Torres},
{\em Numerical methods for backward stochastic differential equations},
Ann. Appl. Probab., 12 (2002), pp. 302-316.

\bibitem {MZ2002}
{\sc J. Ma and J. Zhang},
{\em Representation theorems for backward stochastic differential equations},
Ann. Appl. Probab., 12 (2002), pp. 1390-1418.

\bibitem {MT2006}
{\sc G. N. Milstein and M. V. Tretyakov},
{\em Numerical algorithms for forward-backward stochastic differential equations},
SIAM J. Sci. Comput., 28 (2006), pp. 561-582.

\bibitem {MT2007}
{\sc G. N. Milstein and M. V. Tretyakov},
{\em Discretization of forward-backward stochastic differential equations and related quasi-linear
parabolic equations},
IMA J. Numer. Anal., 27 (2007), pp. 24-44.

\bibitem {OS2011}
{\sc B. \O sendal, A. Sulem, and T. Zhang},
{\em Optimal control of stochastic delay equations and time-advanced backward stochastic differential equations},
Adv. Appl. Probab., 43 (2011), pp. 572-596.

\bibitem {PP1990}
{\sc E. Pardoux and S. Peng},
{\em Adapted solution of a backward stochastic differential equation},
Systems Control Lett., 14 (1990), pp. 55-61.

\bibitem {P1991}
{\sc S. Peng},
{\em Probabilistic interpretation for systems of quasilinear parabolic partial differential equations},
Stochastics Stochastics Rep., 37 (1991), pp. 61-74.

\bibitem {P1999}
{\sc S. Peng},
{\em A linear approximation algorithm using BSDE},
Pacific Economic Review, 4 (1999), pp. 285--291.

\bibitem {PP1994}
{\sc E. Pardoux and S. Peng},
{\em Backward doubly stochastic differential equations and systems of quasilinear SPDEs},
Probab. Theory Relat. Fields, 98, 4 (1994), pp. 209-227.

\bibitem {PY2009}
{\sc S. Peng and Z. Yang},
{\em Anticipated backward stochastic differential equations},
Ann. Probab. 37 (2009), pp. 877-902.


\bibitem {Zh2004}
{\sc J. Zhang},
{\em A numerical scheme for BSDEs},
Ann. Appl. Probab., 14 (2004), pp. 459-488.

\bibitem {ZhCP2006}
{\sc W. Zhao, L. Chen, and S. Peng},
{\em A new kind of accurate numerical method for backward stochastic differential equations},
SIAM J. Sci. Comput., 28 (2006), pp. 1563-1581.

\bibitem {ZhFZh2014}
{\sc W. Zhao, Y. Fu, and T. Zhou},
{\em New kinds of high-order multistep schemes for coupled forward backward stochastic differential equations},
SIAM J. Sci. Comput., 36 (2014), pp. 1731-1751.

\bibitem {ZhLZh2012}
{\sc W. Zhao, Y. Li, and G. Zhang},
{\em A generalized $\theta$-scheme for solving backward stochastic differential equations},
Discrete Contin. Dyn. Syst. Ser. B, 17 (2012), pp. 1585-1603.

\bibitem {ZhWP2009}
{\sc W. Zhao, J. Wang, and S. Peng},
{\em Error Estimates of the $\theta$-Scheme for Backward Stochastic Differential Equations},
Dis. Cont. Dyn. Sys. B, 12 (2009), pp. 905-924.

\bibitem {ZhZhJ2010}
{\sc W. Zhao, G. Zhang, and L. Ju},
{\em A stable multistep scheme for solving backward stochastic differential equations},
SIAM J. Numer. Anal., 48 (2010), pp. 1369-1394.

\bibitem {ZhZhJ2016}
{\sc W. Zhao, W. Zhang, and L. Ju},
{\em A multistep scheme for decoupled forward-backward stochastic differential equations},
Numer. Math. Theory Methods Appl., 9 (2016), pp. 262-288.
\end{thebibliography}
\end{document}